\documentclass[final,3p,times,12pt]{elsarticle}

 \pdfoutput=1 

\usepackage{amssymb}
\usepackage{amsthm}
\usepackage{amsmath}
\usepackage{latexsym}
\usepackage{times}
\usepackage{graphicx}
\usepackage{epsfig}
\usepackage{epstopdf}
\usepackage{mathrsfs}
\usepackage{bm} 

\usepackage[active]{srcltx}
\usepackage{subcaption}

\usepackage{hyperref}  		
\hypersetup{colorlinks=true,citecolor={red},linkcolor={blue},linktocpage} 
\usepackage{color}
\definecolor{armygreen}{rgb}{0.29, 0.33, 0.13}
\definecolor{maroon(x11)}{rgb}{0.69, 0.19, 0.38}
\definecolor{darkorange}{rgb}{1.0, 0.55, 0.0}
\definecolor{auburn}{rgb}{0.43, 0.21, 0.1}

\newtheorem{Theorem}{Theorem}

\newtheorem{Remark}{Remark}

\newtheorem{Lemma}{Lemma}

\newcommand{\uinc}{u_{\rm inc}}

\newcommand{\Sph}{\mathbb{S}}

\usepackage{color}
\definecolor{hotcolor}{rgb}{1,0,0}

\journal{Elsevier}

\begin{document}

\begin{frontmatter}




\title{\textbf{High order methods for acoustic scattering: Coupling Farfield Expansions ABC with Deferred-Correction methods}}

\author[label1]{Vianey Villamizar\corref{cor1}}
\ead{vianey@mathematics.byu.edu}

\author[label1]{Dane Grundvig}
\ead{danesvig@gmail.com}

\author[label2,label3]{Otilio Rojas}
\ead{rojasotilio@gmail.com}

\author[label4]{Sebastian Acosta}
\ead{sacosta@bcm.edu}
\ead[url]{sites.google.com/site/acostasebastian01}

\address[label1]{Department of Mathematics, Brigham Young University, Provo, UT}
\address[label2]{Barcelona Supercomputing Center (BSC)}
\address[label3]{Universidad Central de Venezuela, Facultad de Ciencias, Caracas, Venezuela}
\address[label4]{Department of Pediatrics - Cardiology, Baylor College of Medicine, Houston, TX}
\cortext[cor1]{Corresponding author}

\begin{abstract}
Arbitrary high order numerical methods  for time-harmonic acoustic scattering problems originally defined on unbounded domains are constructed.
This is done by coupling recently developed high order local absorbing boundary conditions (ABCs) with finite difference methods for the Helmholtz equation.
 These ABCs are based on exact representations of the outgoing waves by means of farfield expansions. The finite difference  methods, which are constructed from a deferred-correction (DC) technique, approximate the Helmholtz equation and the ABCs, with the appropriate number of terms, to any desired order. 
 As a result,  high order numerical methods with an overall order of convergence equal to the order of the DC schemes are obtained.
   A detailed construction of these DC finite difference schemes is presented. Additionally, a rigorous proof of the consistency of the DC schemes with the Helmholtz equation and the ABCs in polar coordinates is also given. 
The results of several numerical experiments corroborate the high order convergence of the novel method.
\end{abstract}


\end{frontmatter}

\section{Introduction} \label{Section.Intro}

The propagation and scattering of acoustic waves in the presence of impenetrable obstacles, in an unbounded medium, is an important problem for which significant efforts have been dedicated. However, there are still aspects of this problem that have not yet been satisfactorily solved .  One of them is the construction of easily implementable, reliable and stable high order numerical methods for the accurate approximation of its solution. This is the subject of this work. The construction of high order numerical methods is motivated by the need to obtain highly precise numerical solutions at relatively low computational costs.

 A classical strong formulation, when a time-harmonic incident 
wave, $\uinc$, is scattered from an obstacle with a boundary $\Gamma$ embedded in an unbounded acoustic region $\Omega$,
consists of finding $u\in C^2(\Omega) \cap C({\overline \Omega})$ such that 
 \begin{align}
& \Delta u + k^2 u = f  &\text{in $\Omega$}, \label{BVPsc1} \\
& u = - \uinc \qquad \mbox{or}\qquad \partial_{n} u = - \partial_{n} \uinc ,&\text{on $\Gamma$,} \label{BVPsc2} \\
& \lim_{r \rightarrow \infty} r^{(\delta -1/2)} \left( \partial_{r} u - \mathrm{i} k u \right) = 0,\label{BVPsc3}
\end{align}

\noindent
 where $\Delta$ denotes the Laplace operator, $\partial_n$ is the normal derivative and $i$ is the imaginary unit.
 Both the wavenumber $k$ and the source term $f$ may vary in space. On $\Gamma$, we will study both boundary conditions at the scatterer, either the first equation in (\ref{BVPsc2}) corresponding to the sound-soft Dirichlet condition, or,  the second equation in (\ref{BVPsc2}) corresponding to the sound-hard Neumann condition.
 Equation (\ref{BVPsc3}) is known as the Sommerfeld radiation condition, where $r = |\textbf{x}|$ and $\delta=2$ or 3 for two or three dimensions, respectively.
 This condition renders $u$ as an outgoing wave.  
 
 It is well-known \cite{Bayliss1985, Babuska1999,BabuskaSauter1997, Wang2014} that the accuracy of the numerical methods for the Helmholtz equation (\ref{BVPsc1}) based on finite differences or finite elements deteriorates rapidly when the wave number $k$ increases. This phenomenon is known as pollution error. A common practice, to avoid this error for a given numerical method, consists of increasing the number of points per wavelength $PPW = \lambda/h$, where $\lambda$ is the wavelength and $h$ represents the grid step size. However, this approach becomes computational very costly as $k$ increases. An alternative to alleviate the computational cost is to employ high order schemes since they require less points per wavelength to achieve same accuracy level as their low order counterparts. This is the approach that we follow in this work by constructing an arbitrary high order finite difference method for a bounded version of the acoustic scattering problem (\ref{BVPsc1})-(\ref{BVPsc3}). 

Among the most popular alternatives  to finite difference are finite element (FEM) and boundary element methods (BEM). These methods have their own advantages and
shortcomings when approximating the solutions of a boundary value problem (BVP). An important FEM advantage is their ability to deal with domains of arbitrary shape.
 However, high order convergence usually requires high number of degrees of freedom which normally leads to elevated computational cost. The BEM have the advantage that the Sommerfeld radiation condition is already built into the numerical method, so there is no need to introduce an artificial boundary and define an ABC on it. In contrast, the BEM major shortcoming is that they are limited to homogeneous media. In this study, we opt for finite difference methods because  they are easy of use and their implementation is rather simple and flexible enough to be applied to heterogeneous media.

In the context of finite difference methods, there has been a lot of interest in high order numerical methods in recent years. 
In fact, for interior problems modeled by the Helmholtz equation, several fourth and sixth order numerical methods have appeared in the last 25 years. For instance, Singer and Turkel \cite{Singer-Turkel1998, Singer-Turkel2006} developed compact fourth and sixth order methods in two dimensions for constant wavenumber using cartesian coordinates. 
Sutmann \cite{Sutmann2007} devised a compact sixth order method for  Dirichlet boundary value problems (BVPs)  and Nabavi et al. \cite{Nabavi2007} for Neumann BVPs. 
All of these compact numerical methods were obtained from the so called {\it equation-based} \cite{Singer-Turkel1998} procedure
By applying it, they obtained their compact fourth and sixth order 9-point finite difference formulas to approximate the two-dimensional  Helmholtz equation in cartesian coordinates.
It resembles the 
strategy followed by Collatz and Leveque in  \cite{CollatzBook,LevequeBook}, respectively, to obtain the well-known compact 9-point finite difference formula  for the two-dimensional Poisson equation in cartesian coordinates.  More recently, Zhang et al. \cite{Zhang-Wang-Guo2019} derived a sixth order finite difference scheme for the Helmholtz equation with inhomogeneous Robin boundary conditions in two dimensions. 

Other authors \cite{Kim2003,Britt-Tsynkov-Turkel2011} constructed equation-based compact 9-point fourth and sixth order schemes in cartesian coordinates for interior problems modeled by the two-dimensional Helmholtz equation with variable wavenumber and/or variable coefficients.  Later, Turkel et al. \cite{Turkel-Gordon2013}  also developed a  method for  the three-dimensional Helmholtz equation with variable wavenumber. Compact fourth order finite difference methods have also been devised for the two-dimensional Helmholtz equation with high wavenumbers. For instance in  \cite{Wu2017}, Wu suppressed the numerical dispersion by using nine
points to formulate a compact fourth-order approximation for the term of zero order. Also, Fu in \cite{Fu2008} introduced an alternative compact fourth order method for high frequency, which is independent of the wavenumber. They were able to obtain approximations for wavenumbers as high as $k$ = 500 and 800, respectively.

In recent years, Medvisnky et al. \cite{Medvinsky-Tsynkov-Turkel2012,Medvinsky-Tsynkov-Turkel2013,Medvinsky-Tsynkov-Turkel2016, Medvinsky-Tsynkov-Turkel2019} extended the {\it method of difference potentials}, introduced by Ryaben'kii for standard centered finite difference schemes \cite{Ryabenkii1985,Ryabenkii2002}, to compact finite difference schemes.  This procedure involves several steps inspired in the theory of Calderon's operator for partial differential equations.
It consists of reducing the Helmholtz equation from its domain $\Omega$ 
to an equivalent equation defined only on its boundary $\Gamma$ and numerically solve this simpler equation. As part of this process, the Helmholtz equation is approximated by a compact fourth or sixth order scheme.
A detailed account of this procedure can be found in \cite{Medvinsky-Tsynkov-Turkel2012,Britt-Tsynkov-Turkel2013,Medvinsky-Tsynkov-Turkel2019}.
At the final stage of the computation, they calculate a grid function $\xi_{\gamma}$  from which a discrete Calderon's potential or difference potentials is obtained.
 They show that this difference potential approximates its continuous counterpart with the same order of accuracy of a compact scheme used to approximate the Helmholtz equation \cite{Medvinsky-Tsynkov-Turkel2019}.
As a consequence, the approximation of the scattered field $u^h$ also converges to $u$  with this same order of accuracy. 

Among the advantages of the method of difference potentials is its ability to handle smooth curvilinear boundaries and variable wavenumbers. 
The grid function ${\xi}_{\gamma}$ is represented in terms of a basis for the space of smooth functions on $\Gamma$, which is evaluated at the grid points. 
Therefore, the linear system arising from the discretization has as unknowns the expansion coefficients with respect to the chosen basis, instead
of the node values of $\xi_{\gamma}$.
A possible disadvantage is
that
the boundary conditions should be represented  by a volumetric spectral solver for the same basis to complete the linear system. This certainly leads to a more dense linear system than those obtained from direct application of finite difference or finite element techniques. However, a QR decomposition technique may work well in this case.

The above  methodology was applied first  to interior problems modeled by the two-dimensional Helmholtz equation with variable wavenumber \cite{Medvinsky-Tsynkov-Turkel2012,Britt-Tsynkov-Turkel2013} employing compact fourth and sixth order  equation-based schemes. For smooth solutions, the expected fourth and sixth order convergence were achieved. Later, the same authors extended the method of difference potentials to exterior problems. In fact in \cite{Medvinsky-Tsynkov-Turkel2013,Medvinsky-Tsynkov-Turkel2016}, two-dimensional transmission and scattering problems were solved for simple-shaped obstacles  and smooth regions, respectively. 
 In both cases, a fourth order discrete ABC, first introduced in  \cite{Britt-Tsynkov-Turkel2010} and
 defined in the Fourier space, was employed. It was combined with the method of difference potentials consisting of a compact fourth order accurate finite difference scheme. The fourth order convergence of the numerical solution to the exact solution was verified in several experiments. However, the discrete nature of this ABC limits its use to more general problems and makes its extension to higher orders difficult. 

Recently, an arbitrary high order three-dimensional ABC in spectral form was devised for exterior problems modeled by the Helmholtz equation in \cite{Medvinsky-Tsynkov-Turkel2019}. This was elegantly coupled with a sixth order interior scheme obtained from the method of difference potentials. It was applied only to radiating source problems (monopole and dipole), but scattering problems were not attempted. For these problems, the sixth order convergence  was experimentally corroborated. Although this novel ABC can be implemented at arbitrary high orders, direct coupling with more popular finite element or finite difference methods is not possible given its current spectral formulation. 

As described in the previous paragraphs, there have been numerous attempts to construct high order finite difference schemes for the Helmholtz equation. Similarly, the derivation of high order local ABC for time-harmonic acoustic scattering problems has been intensively pursued by many researchers, since the pioneer work of Bayliss-Gunzburger-Turkel  (BGT) \cite{Bayliss01}. For example, Zarmi and Turkel \cite{Zarmi-Turkel} developed an annihilating technique that can be applied to rather general series representation of the solution in the exterior of the computational domain. As a result, they were able to obtain high order local ABC without derivative terms greater than order two for exterior problems in the plane.  Also, Rabinovich et al. \cite{Rab-Giv-Bec-2010} adapted the auxiliary variable 
 formulation of local high order ABC for the wave equation of Hagstrom-Warburton (H-W) \cite{H-W} to time-harmonic problems in a waveguide and a quarter plane modeled by the Helmholtz equation. More recently, Hagstrom and Kim \cite{Hagstrom-Kim2019} adapted an improved version of H-W called {\it complete radiation boundary conditions} to waveguides problems in the frequency domain. They solved radiation problems inside semi-infinite waveguides with sources in their finite west boundaries. In principle, the adapted H-W absorbing boundary condition can be implemented for time-harmonic exterior problems in the entire plane. However, exterior problems are not included in \cite{Rab-Giv-Bec-2010,Hagstrom-Kim2019}.
 The application of H-W type ABC for the exterior problems use  rectangular artificial boundaries
 to enclose the scatterers. As a consequence, special treatment  at corners formed by the intersection of two flat segments is required.
 In \cite{Rab-Giv-Bec-2010,Hagstrom-Givoli2014}, 
 the authors acknowledge that these corner conditions are quite involved and even difficult to devise.  In another recent publication Duhamel \cite{Duhamel2020} constructs a high order ABC at the discrete level for a radiation problem from a circular obstacle. The results compare favorably with second order ABCs such as BGT and Feng's. However, a convergence analysis is not  included to establish more clearly the advantages of this technique.
 In all these works \cite{Zarmi-Turkel,Rab-Giv-Bec-2010,Hagstrom-Kim2019,Duhamel2020}, their high order ABCs are coupled with low order discretization schemes for the interior domain based on bilinear finite elements. As a consequence, overall low order  numerical methods (at most second order)  are obtained.
For other contributions on high order ABC, the reader is referred to the review article \cite{GivoliReview2} and also to the introduction in  \cite{JCP2017}.

 An efficient alternative to high order local ABC is provided by a technique called {\it perfectly matched layer} (PML). This consists of surrounding the artificial boundary with a layer of elements where the Helmholtz equation is modified. The PML was first introduced by Berenger for electromagnetic waves in the highly cited paper \cite{Berenger01}. An adaptation to the Helmholtz equation was devised by Becache et al. in \cite{Becache2004}. 
Good choices of the layer's size and the parameters of the absorbing layers lead to excellent absorption of waves. However, PML tends to be very sensitive to the choice of the computational parameters. Also, it is hard to establish a clear notion of convergence.  For the interested reader, a  good comparison of the two approaches, high order local ABC and PML, is given in \cite{Rab-Giv-Bec-2010}.

As far as the authors know, overall high order finite difference methods for exterior time-harmonic acoustic scattering has only been constructed up to fourth order \cite{Britt-Tsynkov-Turkel2010,Medvinsky-Tsynkov-Turkel2013,Medvinsky-Tsynkov-Turkel2016}. 
In the present study, we develop  arbitrary high  order finite difference schemes for the Helmholtz equation based on a {\it deferred-correction} (DC) methodology (see \cite{LevequeBook} Section 3.5).
 Among the pioneer applications of deferred-corrections to differential equations are the works by Pereyra \cite{Pereyra1968,Pereyra1970}. 
Our construction proceed by coupling arbitrary high order DC finite difference schemes for the Helmholtz equation with high order DC finite difference schemes corresponding to arbitrary high order ABCs based on {\it farfield expansions}, which were developed by Villamizar et al. in \cite{JCP2017}. As a result of combining these high order techniques (domain's interior and boundary), we obtain an overall arbitrary high order method for acoustic scattering. Preliminary results were presented in \cite{Waves2017}. 
Of course, the arbitrary high order property of this method is limited by the computer arithmetic and the computer resources available.
 The construction and performance analysis of this overall and arbitrary high order finite difference method for acoustic scattering problems is discussed in detail in the following sections.

\section{The two-dimensional scattering BVP with Karp's farfield expansion absorbing boundary condition (KFE-BVP)}
\label{Section.ABC2D}

 The exterior problem (\ref{BVPsc1})-(\ref{BVPsc3}) needs to be reformulated as an equivalent BVP on a bounded domain before a numerical scheme, based on volume discretization methods, can be applied. 
In Villamizar et al. \cite{JCP2017}, such problem transformation was carried out by introducing a circular (2D) or spherical (3D) artificial boundary $S$ that enclose all the scatterers, regarding of their particular shapes, and then by defining high order local ABCs on these artificial boundaries. Their definition is based on  the following series representations of the exact solution $u$ outside the region bounded by the artificial boundary $S.$
 \begin{enumerate}
\item
\mbox{Karp's farfield expansion in two dimensions [Karp1961]:}
\begin{equation}
u(r,\theta)=H_0(kr) \sum_{l=0}^{\infty} \frac{F_l(\theta)}{(kr)^l} + H_1(kr)\sum_{l=0}^{\infty} \frac{G_l(\theta)}{(kr)^l},
\qquad r\ge R, \label{Karp}
\end{equation}
\item
\mbox{Wilcox's farfield expansion in three dimensions [Wilcox1956]:}
\begin{eqnarray}
u(r,\theta,\phi) = \frac{e^{ikr}}{kr}\sum_{l=0}^{\infty} \frac{F_l(\theta,\phi)}{(kr)^l} \qquad\qquad\qquad\qquad\qquad\quad r\ge R.\label{Wilcox}
\end{eqnarray}
\end{enumerate}
In the series (\ref{Karp}), $r$ and $\theta$ are polar coordinates. The functions $H_0$ and $H_1$ are Hankel functions of first kind of order 0 and 1, respectively.
The coefficients $F_l(\theta)$ and $G_l(\theta)$ ($l>1$) can be determined from $F_0(\theta)$ and $G_0(\theta)$ by the recursion formulas
\begin{align}
& 2 l G_{l}(\theta) = (l-1)^2 F_{l-1}(\theta) + d^2_{\theta}  F_{l-1}(\theta) , \qquad && \text{for $l=1,2, \dots$} \label{RRecurrence1}\\
& 2 l F_{l}(\theta) = - l^2 G_{l-1}(\theta) - d^2_{\theta} G_{l-1}(\theta), 
\qquad && \text{for $l=1,2, \dots$}. \label{RRecurrence2}
\end{align}
In the series (\ref{Wilcox}), $r$, $\theta$, and $\phi$ are spherical coordinates and $\Delta_{\Sph}$ is the Laplace-Beltrami operator in the angular coordinates $\theta$ and $\phi$.  Also, the coefficients $F_l$ ($l \ge 1$) can be determined by the recursion formula,
\begin{eqnarray}
2 i l F_l(\theta,\phi) = l(l-1) F_{l-1}(\theta,\phi) + \Delta_{\Sph} F_{l-1}(\theta,\phi), \qquad l \geq 1. \label{AW-Recursive}
\end{eqnarray}

The artificial boundary $S$ divides the domain into a bounded computational region $\Omega^-$ enclosed by the obstacle boundary $\Gamma$ and the artificial boundary $S$, and the exterior unbounded region $\Omega^+ = \Omega\backslash \bar {\Omega}^-$. 
Once this decomposition of the domain is done, the original unbounded problem in $\Omega$ is reformulated as a bounded problem in $\Omega^-$ by matching the solution $u$ inside $\Omega^-$ with the semi-analytical representation of the solution $u$ in $\Omega^+$ given by the series representations (\ref{Karp})-(\ref{Wilcox}). 
These series are uniformly and absolutely convergent for $r>R$. They can be differentiated  term by term with respect to $r$, $\theta$, and $\phi$ any number of times and the resulting series all converge absolutely and uniformly. The angular functions $F_l$ and $G_l$ become additional unknowns of the new bounded BVP. They depend on the geometry of the scatterers and the physical properties of the medium inside the computational region  
$\Omega^-.$

 In this work, we specialize in the two-dimensional case with $f=0$ for simplicity, but its extension to non-homogeneous and three-dimensional scattering problems follows a very similar procedure. In \cite{JCP2017}, a detailed formulation of of a truncated version of an equivalent bounded BVP to (\ref{BVPsc1})-(\ref{BVPsc3}) for the scattered field $u$ and the angular functions $F_l$ and $G_l$ in $\Omega^-$ was introduced as
 \begin{align}
& \Delta u + k^2 u = 0, \quad\qquad  && \mbox{in}\quad \Omega^-, \label{BVPBd1} \\
& u = - \uinc, \qquad\qquad\,\mbox{or} \qquad\qquad \partial_{r} u = -\partial_{r} \uinc, && \mbox{in}\quad \Gamma, \label{BVPBd2} \\
& u(R,\theta)=H_0(kR) \sum_{l=0}^{L-1} \frac{F_l(\theta)}{(kR)^l} + H_1(kR)\sum_{l=0}^{L-1} \frac{G_l(\theta)}{(kR)^l},\label{BVPBd3} \\
&\partial_{r} u(R,\theta) = \partial_{r}\left( H_0(kr) \sum_{l=0}^{L-1} \frac{F_l(\theta)}{(kr)^l} + H_1(kr)\sum_{l=0}^{L-1} \frac{G_l(\theta)}{(kr)^l}\right) \bigg|_{r=R},
\label{BVPBd4} \\
& \partial_{r}^2 u(R,\theta) = \partial_{r}^2\left( H_0(kr)\sum_{l=0}^{L-1} \frac{F_l(\theta)}{(kr)^l} + H_1(kr)\sum_{l=0}^{L-1} \frac{G_l(\theta)}{(kr)^l}\right) \bigg|_{r=R}, \label{BVPBd5}\\
& 2 l G_{l}(\theta) = (l-1)^2 F_{l-1}(\theta) + d^2_{\theta}  F_{l-1}(\theta) , \qquad && \text{for $l=1,2, \dots L-1$} \label{Recurrence1}\\
& 2 l F_{l}(\theta) = - l^2 G_{l-1}(\theta) - d^2_{\theta} G_{l-1}(\theta), \qquad && \text{for $l=1,2, \dots L-1$}. \label{Recurrence2}
\end{align}
where $\Delta$ represents the Laplacian operator and $R$ is the radius of the circular artificial boundary $S$. The equations (\ref{BVPBd3})-(\ref{BVPBd5}) for the truncated Karp's expansion, with 
 $F_l$ and $G_l$ ($l=0\dots L-1$) unknown angular functions, supplemented by the recurrence formulas (\ref{RRecurrence1})-(\ref{RRecurrence2}) constitute the novel \textit{Karp's farfield expansion} ABC (KFE) constructed  in \cite{JCP2017}. 
 A careful consideration on the number of unknowns of the BVP  at the artificial boundary $r=R$ reveals the need of having as many equations as (\ref{BVPBd3})-(\ref{Recurrence2}) defining the ABC. In fact, the number of unknowns  at the artificial boundary are 
 $u(R,\theta)$,  $F_0(\theta)$, $G_0(\theta)$, $F_l(\theta)$,  and $G_l(\theta)$ ($l=1\dots L-1$).  They are  $3+2(L-1)$ in total which is the same number of independent equations given by (\ref{BVPBd3})-(\ref{Recurrence2}).

 It was shown in \cite{JCP2017}, that the numerical solution of  (\ref{BVPBd1})-(\ref{Recurrence2}) exhibits second order convergence to the exact solution, by using second order finite difference methods to approximate the Helmholtz equation in $\Omega^-$ as well as the various equations for the absorbing boundary condition  at the artificial boundary. This result was obtained by employing relatively few  terms (usually, from three to eight terms) in the Karp's expansion. Our main purpose in this article is to further exploit the high order accuracy of the KFE, by coupling high order discretizations of it to interior high order finite difference appoximations for Helmholtz equation, which leads to overall high order numerical methods for acoustic scattering. 
In the following sections, these novel high order numerical methods for (\ref{BVPBd1})-(\ref{Recurrence2}) are derived.

\section{Derivation of high order DC methods for the KFE-BVP}  \label{ssec:3}


In the next subsections, a detailed formulation of the fourth order DC numerical scheme is given for the Helmholtz equation in the domain $\Omega^{-}$ bounded externally by the artificial boundary $S$ of circular shape with radius $R$. Similarly, we also develop a fourth order DC scheme for the approximation of the high order KFE imposed on the artificial boundary. This is followed by a formulation of these numerical schemes (interior and artificial boundary)  of arbitrary order $p$.

\subsection{Fourth order DC scheme for the Helmholtz equation in polar coordinates} \label{ssec:3.1}

First we consider that the domain $\Omega^{-}$ can be covered by a polar grid with constant radial and angular steps $\Delta r$ and $\Delta\theta$, respectively. The number of grid points in the radial and angular directions is $N, m > 1$, respectively. For a given grid point $(r_i,\theta_j)$, the discrete value of the scattered field is denoted by $u_{i,j}=u(r_i,\theta_j)$. Notice that the pairs $(r_i,\theta_1)$ and $(r_i,\theta_{m+1})$ represent the same physical point due to periodicity in the angular direction, thus $u_{i,1}=u_{i,m+1}$ for $i \leq N$. Thus, the grid supports $N \times m$ wavefield evaluations. 

We start with the standard centered second order finite difference method for the Helmholtz equation in polar coordinates given by
\begin{align}
 {\mathcal{H}}^2_5U^2_{ij}&\equiv \frac{U^2_{i+1,j} - 2 U^2_{i,j} +U^2_{i-1,j}}{\Delta r^2} +
\frac{1}{r_i}\frac{U^2_{i+1,j} - U^2_{i-1,j}}{2\Delta r} 
+ \frac{1}{r_i^2}\frac{U^2_{i,j+1} - 2 U^2_{i,j} +U^2_{i,j-1}}{\Delta \theta^2} + k^2 U^2_{ij}=0. \label{2ndorder}
 \end{align}
The symbol  $U^2_{ij}$ is used to describe a discrete solution of (\ref{2ndorder}). 
The subindex 5 of $ {\mathcal{H}}^2_5$ is used to acknowledge that this finite difference formula consists of a 
 5-point stencil. The super-index 2 states that ${\mathcal{H}}^2_5U^2_{ij}=0$ is a consistent second order finite difference approximation of the Helmholtz equation (\ref{BVPBd1}).  
 Also, we consider a discrete function $U^2_{ij}$ which is a second order approximation to the exact solution $u$ of the Helmholtz equation (\ref{BVPBd1}) subject to the boundary conditions (\ref{BVPBd2})-(\ref{Recurrence2}), i.e., 
 \begin{equation}
U^2_{ij} =u(r_i,\theta_j) + \Delta r^2 v(r_i,\theta_j) + \Delta \theta^2 w(r_i,\theta_j)=u(r_i,\theta_j)+h^2z(r_i,\theta_j),\label{SecOrderU}
\end{equation}
where  $v$, and $w$ and $z$ are sufficiently smooth bounded functions on the closure of $\Omega^{-}$ and $h=\max\{\Delta r,\Delta \theta\}.$
The computation of  $U^2_{ij}$ is fully described  in \cite{JCP2017}. 
 
  Applying $\mathcal{H}^2_5$ to  $u$ and evaluating it at $(r_i,\theta_j)$ leads to
 \begin{align}
 {\mathcal{H}}^2_5u_{ij}&=
\frac{u_{i+1,j} - 2 u_{i,j} +u_{i-1,j}}{\Delta r^2} +
\frac{1}{r_i}\frac{u_{i+1,j} - u_{i-1,j}}{2\,\Delta r} 
+\frac{1}{r^2} \frac{u_{i,j+1} - 2 u_{i,j} +u_{i,j-1}}{\Delta \theta^2} + k^2 u_{ij}\nonumber\\ 
&=  \left(\Delta_{r\theta} u + k^2 u\right)_{ij} +
 \frac{\Delta r^2}{12}\left(\left(u_{4r}\right)_{ij} + \frac{2}{r_i}\left(u_{3r}\right)_{ij}\right) +
\frac{\Delta \theta^2}{12\,r_i^2}\left(u_{4\theta}\right)_{ij}  
+ O(\Delta r^4) + O(\Delta \theta^4). \label{DC4Exact}
 \end{align}
We seek to obtain a fourth order finite difference scheme for the Helmholtz equation by subtracting the second order leading terms of the truncation error
in the right hand side of (\ref{DC4Exact}) from the second order standard scheme (\ref{2ndorder}). This is followed by substitution of the partial derivatives of  $u$, present in these leading terms, by second order finite difference operators of these partial derivatives of u. They  act on the previously computed discrete solution $U^2_{ij}$ of the standard second order scheme (\ref{2ndorder}) which approximates $u$ to second order. This is the fundamental idea in the formulation of the  DC method proposed in this work for the Helmholtz equation.  
More precisely, the construction of the fourth order DC technique for the Helmholtz equation  (\ref{BVPBd1}) subject to the boundary conditions (\ref{BVPBd2})-(\ref{Recurrence2})
consists of two steps:
\begin{enumerate}
\item[Step 1:] {\it Obtaining a second order approximation $U^2_{ij}$ to the solution $u$, of the original BVP
 (\ref{BVPsc1})-(\ref{BVPsc3})}.\\
 Approximate the Helmholtz equation  (\ref{BVPBd1}) by  the standard centered second order 5-point stencil scheme $ {\mathcal{H}}^2_5U^2_{ij}=0$, defined in (\ref{2ndorder}).
Also, use appropriate one-sided second order schemes to approximate all the other boundary differential operators contained in the boundary conditions (\ref{BVPBd2})-(\ref{Recurrence2}). Then, by solving the linear system that ultimately results, from the discretization of all the equations of the KFE-BVP, obtain a second order numerical approximation $U^2_{ij}$ to the exact solution $u$, of the original BVP. This computation is done in the article \cite{JCP2017} where the KFE condition was first introduced.

\item[Step 2:]  {\it Formulation of the new fourth order finite difference DC numerical scheme for the Helmholtz equation in terms of the $U^2_{ij}$ obtained in step 1.
}\\
The second step consists of approximating the  
continuous derivatives 
$u_{4r}$, $u_{3r}$, and $u_{4\theta}$ in (\ref{DC4Exact}) using standard centered second order finite differences acting on $U^2_{ij}$ as follows,
\begin{align}
&\left(u_{4r}\right)_{ij}\approx D_{4r}^2U^2_{ij}
\equiv 
\frac{1}{\Delta r^4}\left[
 U^2_{i-2,j} - 4 U^2_{i-1,j} + 6 U^2_{i,j}
-4U^2_{i+1,j} + U^2_{i+2,j} 
\right], \label{D4r2}
\\
&\left(u_{3r}\right)_{ij}\approx D_{3r}^2U^2_{ij}
\equiv 
\frac{1}{\Delta r^3}\left[
-\frac{1}{2} U^2_{i-2,j} + U^2_{i-1,j} 
-U^2_{i+1,j} + \frac{1}{2}U^2_{i+2,j} 
\right], \label{D3r2}
\\
&\left(u_{4\theta}\right)_{ij}\approx D_{4\theta}^2U^2_{ij}
\equiv 
\frac{1}{\Delta \theta^4}\left[
 U^2_{i,j-2} - 4 U^2_{i,j-1} + 6 U^2_{i,j}
-4U^2_{i,j+1} + U^2_{i,j+2} 
\right].\label{D4theta}
\end{align}
We use the notation $D^p_{qr}$ to designate the $p$th order centered finite difference discrete operator of the $q$th derivative with respect to $r$. Analogously, $D^p_{q\theta}$ designates the $p$th order centered finite difference operator of the $q$th order derivative with respect to $\theta$.
Then by substituting (\ref{D4r2})-(\ref{D4theta}) into (\ref{DC4Exact}),  we arrive to the new fourth order finite difference DC numerical scheme for the Helmholtz equation given by
\begin{align}
 {\mathcal{H}}_5^4U^4_{ij} &\equiv \frac{U^4_{i+1,j} - 2 U^4_{i,j} +U^4_{i-1,j}}{\Delta r^2} +
\frac{1}{r_i}\frac{U^4_{i+1,j} - U^4_{i-1,j}}{2\Delta r} 
+ \frac{1}{r_i^2}\frac{U^4_{i,j+1} - 2 U^4_{i,j} +U^4_{i,j-1}}{\Delta \theta^2} + k^2 U^4_{ij}\nonumber\\
&-\frac{\Delta r^2}{12}\left(D_{4r}^2U^2_{ij} + 
\frac{2}{r_i}D_{3r}^2U^2_{ij}\right) - 
\frac{\Delta \theta^2}{12r_i^2}D_{4\theta}^2U^2_{ij}\nonumber\\
&= D_{2r}^2 U^4_{ij} + \frac{1}{r_i} D_{r}^2 U^4_{ij} + \frac{1}{r_i^2} D_{2\theta}^2 U^4_{ij} + k^2 U^4_{ij} \nonumber \\
&\quad -\frac{\Delta r^2}{12}\left(D_{4r}^2U^2_{ij} + 
\frac{2}{r_i}D_{3r}^2U^2_{ij}\right) - 
\frac{\Delta \theta^2}{12r_i^2}D_{4\theta}^2U^2_{ij}=
0.\label{4thorder}
 \end{align}
\end{enumerate}
Notice that the new finite difference scheme (\ref{4thorder}) for the unknown discrete function $U^4_{ij}$  consists of the same 5-point stencil of the standard centered second order scheme. The difference is that (\ref{4thorder}) has an additional known term given by 
$$ \frac{\Delta r^2}{12}\left(D_{4r}^2U^2_{ij} + 
\frac{2}{r_i}D_{3r}^2U^2_{ij}\right) +
\frac{\Delta \theta^2}{12r_i^2}D_{4\theta}^2U^2_{ij}, $$
which is calculated from the second order numerical solution $U^2_{ij}$ of $u$ already computed in the first step.

\begin{Remark}
At the artificial boundary $r = R$, we use appropriate second order one-sided finite differences acting on $U^2_{Nj}$ to approximate the various derivatives present in the leading terms of the truncation error in (\ref{DC4Exact}). They are defined in the following section.
\end{Remark}

In what follows, we state and prove our claim that the finite difference scheme (\ref{4thorder}) is a fourth order approximation of the Helmholtz equation.
\medskip 

\begin{Theorem}\label{Th1}
The new DC numerical scheme (\ref{4thorder}), or equivalently, 
\begin{align}
 {\mathcal{H}}_5^4{U^4}_{ij}\equiv {\mathcal H}_5^2 U^4_{ij} 
-\frac{\Delta r^2}{12}\left(D_{4r}^2U^2_{ij} + 
\frac{2}{r_i}D_{3r}^2U^2_{ij}\right) - 
\frac{\Delta \theta^2}{12r_i^2}D_{4\theta}^2U^2_{ij} = 0\label{4thorderT}
 \end{align}
is a consistent finite difference approximation of the Helmholtz equation in polar coordinates 
$$ \Delta_{r\theta} u + k^2 u
= u_{rr} + \frac{1}{r} u_r + \frac{1}{r^2}u_{\theta\theta} +k^2 u = 0,$$
 of order $\mathcal{O}(\Delta r^4) +\mathcal {O}(\Delta \theta^4) + \mathcal{O}(\Delta r^2\Delta \theta^2)$ on $\Omega^{-}$,
if the continuous function $u$ has derivatives of order 4 in its two variables $r$ and $\theta$ on $\Omega^-$; and if 
the discrete function ${ U^2_{ij}}$ is a second order approximation of $u$, i.e.,
there exists $v(r,\theta)$, and $w(r,\theta)$ sufficiently smooth and bounded functions on the closure of $\Omega^{-}$ such that 
\begin{equation}
U^2_{ij} =u(r_i,\theta_j) + \Delta r^2 v(r_i,\theta_j) + \Delta \theta^2 w(r_i,\theta_j) = u_{ij} + \Delta r^2 v_{ij} + \Delta \theta^2 w_{ij}.
\label{Uhat}
\end{equation}

\end{Theorem}

Before proving this theorem, we will proof the following lemma.
\begin{Lemma}\label{Lemma1}
For $U^2_{ij}$ and $u(r,\theta)$ satisfying the hypotheses of Theorem \ref{Th1}, it holds that the standard second order centered finite difference operators:
\begin{enumerate}
\item $D_{4r}^2$, as defined in (\ref{D4r2}), acting on $U^2_{ij}$, is consistent with the derivative $u_{4r}$ of order 
$\mathcal{O}(\Delta r^2) + {O}(\Delta \theta^2)$.

\item $D_{3r}^2$, as defined in (\ref{D3r2}), acting on $U^2_{ij}$,
is consistent with the derivative $u_{3r}$ of order 
$\mathcal{O}(\Delta r^2) + \mathcal{O}(\Delta \theta^2)$

\item $D_{4\theta}^2$, as defined in (\ref{D4theta}), acting on $U^2_{ij}$,
is consistent with the derivative $u_{4\theta}$ of order 
$\mathcal{O}(\Delta \theta^2) + \mathcal{O}(\Delta r^2)$
\end{enumerate}
\end{Lemma}
%
%

\begin{proof}
To prove (a), we apply  $D_{4r}^2$ to $U^2_{ij}$ replaced by (\ref{Uhat}). This leads to
\begin{align*}
D_{4r}^2U^2_{ij} &= D_{4r}^2 u_{ij} + \Delta r^2 D_{4r}^2 v_{ij} + \Delta \theta^2 D_{4r}^2 w_{ij}
\\
&= (u_{4r})_{ij} + \mathcal{O}(\Delta r^2) + \Delta r^2 (v_{4r})_{ij} +  \mathcal{O}(\Delta r^4) +  \Delta \theta^2 (w_{4r})_{ij}  
       + \mathcal{O}(\Delta \theta^2 \Delta r^2)
\end{align*}
Therefore, 
\begin{align*}
D_{4r}^2U^2_{ij} 
= (u_{4r})_{ij} + \mathcal{O}(\Delta r^2) +  \mathcal{O}(\Delta \theta^2) 
\end{align*}
and part (a) of the lemma is proved.
The proofs of parts (b) and (c) are completely analogous.
\end{proof}

\begin{proof} (Proof of Theorem \ref{Th1})

First, we rewrite $\mathcal{H}_5^4 U^4_{ij}$ as 
\begin{align}
 \mathcal{H}_5^4 U^4_{ij} = &
 \left( D_{2r}^2 U^4_{ij} - \frac{\Delta r^2}{12}D_{4r}^2U^2_{ij} \right) +
   \frac{1}{r_i}\left(D_{r}^2 U^4_{ij} - \frac{\Delta r^2}{6}D_{3r}^2U^2_{ij} \right) \nonumber\\
 &+ \frac{1}{r_i^2}\left(D_{2\theta}^2 U^4_{ij} 
- \frac{\Delta \theta^2}{12}D_{4\theta}^2U^2_{ij} \right) + k^2 { \bar U}_{ij}.\nonumber
 \end{align}
Next, we apply  $ \mathcal{H}_5^4$ to 
 $u$ satisfying (\ref{Uhat}).
 This is followed by expressing each of the discrete derivatives of $u$ in terms of their corresponding continuous derivatives plus their leading order truncation errors, which leads to
 \begin{align*}
\mathcal{H}_5^4 u_{ij} &= 
\left((u_{rr})_{ij} + \frac{\Delta r^2}{12}\left(u_{4r}\right)_{ij} + 
 {\mathcal O}(\Delta r^4) - \frac{\Delta r^2}{12}D_{4r}^2U^2_{ij} \right)
\\
&\quad+ \frac{1}{r_i}\left((u_{r})_{ij} + \frac{\Delta r^2}{6}\left(u_{3r}\right)_{ij} 
+{\mathcal O}(\Delta r^4) - \frac{\Delta r^2}{6}D_{3r}^2U^2_{ij} 
\right)\\
&\quad+ \frac{1}{r_i^2} \left( (u_{\theta\theta})_{ij} + \frac{\Delta \theta^2}{12}\left(u_{4\theta}\right)_{ij} 
+ {\mathcal O}(\Delta \theta^4) - \frac{\Delta \theta^2}{12}D_{4\theta}^2U^2_{ij}  \right)
+ k^2 { u}_{ij}.
\end{align*}
Reordering the righthand side terms yields
\begin{align*}
 \mathcal{H}_5^4 u_{ij}& =(u_{rr})_{ij} + \frac{1}{r_i}(u_{r})_{ij} +\frac{1}{r_i^2} (u_{\theta\theta})_{ij} +k^2 u_{ij} 
  - \frac{\Delta r^2}{12} \left(  D_{4r}^2U^2_{ij}  -   \left(u_{4r}\right)_{ij} \right) \\
  &- \frac{\Delta r^2}{6} \left(  D_{3r}^2U^2_{ij}  -   \left(u_{3r}\right)_{ij} \right) 
  - \frac{\Delta \theta^2}{12} \left(  D_{4\theta}^2U^2_{ij}  -   \left(u_{4\theta}\right)_{ij} \right) 
+ {\mathcal O}(\Delta r^4) + {\mathcal O}(\Delta \theta^4).
\end{align*}
Then, by applying the statements of Lemma \ref{Lemma1} to the above expression, we get

\begin{align*}
 \mathcal{H}_5^4 u_{ij} = 
  \left(\Delta_{r\theta} u + k^2 u\right)_{ij}
  + \mathcal{O}(\Delta r^4)  + \mathcal{O}(\Delta \theta^4)  + {\mathcal O}(\Delta r^2\Delta \theta^2),
  \end{align*}
which finishes the proof.
\end{proof}

Therefore, the new numerical scheme (\ref{4thorder}) approximates the Helmholtz equation to fourth order, while maintaining a 5-point stencil on the unknown discrete function $U^4_{ij}$. We show in Section \ref{Section.Numerics2D} that there are important savings in the required storage and computational time compared to the 9-point standard centered fourth order finite difference approximation of the Helmholtz equation. This is one of the virtues of the DC technique developed in this work.

\subsection{Fourth order DC approximation at the artificial boundary} \label{ssec:3.2}

At the absorbing boundary, the radial derivatives of the scattered field $u$, present in (\ref{BVPBd4}) and (\ref{BVPBd5}), 
are approximated using standard centered second order finite differences. Then, to increase their accuracy to fourth order via DC, we subtract from these standard finite differences their leading order truncation error  terms. Imitating the previous procedure employed for the Helmholtz equation at the interior points, these leading  order terms are approximated using a previously calculated second order numerical solution $U^2_{ij}$ of the exact solution $u$ of the BVP (\ref{BVPBd1})-(\ref{Recurrence2}).
As a result, we obtain the following discrete non-homogeneous equations at the artificial boundary $r=R$:
\begin{align}
&\frac{U^4_{N+1,j} - U^4_{N-1,j}}{2\Delta r} -\,\,\partial_{r} \left( H_0(kr)\sum_{l=0}^{L-1} 
\frac{F^4_{lj}}{r^l} + H_1(kr)\sum_{l=0}^{L-1} \frac{G^4_{lj}}{r^l}\right)_{r=r_N=R} 
=\frac{\Delta r^2}{6} 
(Dl)_{3r}^2U^2_{N,j} \label{BC44NH}\\
&\frac{U^4_{N+1,j} - 2U^4_{Nj}+U^4_{N-1,j}}{\Delta r^2} - \,\,\partial_{r}^2\left( H_0(kr)\sum_{l=0}^{L-1} \frac{F^4_{lj}}{r^l} + H_1(kr)\sum_{l=0}^{L-1} \frac{G^4_{lj}}{r^l}\right)_{r=r_N=R}
= \frac{\Delta r^2}{12}(Dl)_{4r}^2 U^2_{N,j}, \label{BC54NH}
\end{align}
The forcing terms in (\ref{BC44NH}) and (\ref{BC54NH}) are defined from one-sided second order finite difference approximations $(Dl)_{3r}^2U^2_{N,j}$ and $(Dl)_{4r}^2 U^2_{N,j}$
of $\left(u_{3r}\right)_{Nj}$ and $\left(u_{4r}\right)_{Nj}$, respectively. More precisely,
\begin{align}
&(Dl)_{3r}^2U^2_{N,j} \equiv \frac{1}{\Delta r^3}\left[ \frac{1}{2}U^2_{N-3,j} - 3U^2_{N-2,j} +6U^2_{N-1,j}
-5U^2_{N,j} + \frac{3}{2}U^2_{N+1,j}
\right] \label{SideDer1}\\
&(Dl)_{4r}^2U^2_{N,j} \equiv \frac{1}{\Delta r^4}\left[ -U^2_{N-4,j} + 6 U^2_{N-3,j} - 14U^2_{N-2,j} +16U^2_{N-1,j}
-9U^2_{N,j} + 2U^2_{N+1,j}
\right]. \label{SideDer2}
\end{align}
Notice that the equations (\ref{BC44NH})- (\ref{BC54NH}) involve values of the discrete approximations, $U^4_{ij}$ and $U^2_{ij}$, at the  ghost points $(r_{N+1},\theta_j)$. The unknowns $U^4_{N+1,j}$ also appear in the fourth order approximation of the Helmholtz equation (\ref{4thorder}) evaluated at $i=N$ for $j=1\dots m$. We eliminate it by solving for $U^4_{N+1,j}$ in (\ref{BC44NH}) and substituting it into (\ref{BC54NH}), and into (\ref{4thorder}) evaluated at $i=N$. 
A similar procedure is employed to eliminate the 
ghost values $U^2_{N+1,j}$ during the first step, i.e., as part of the numerical solution of the second order scheme of the KFE-BVP (\ref{BVPBd1})-(\ref{Recurrence2}).

Analogously, we construct deferred-correction fourth order recursion formulas by keeping 
the second order terms of the truncation errors obtained by approximating  the angular derivatives in (\ref{Recurrence1})-(\ref{Recurrence2}) using second order centered finite differences. In fact,
\begin{align}
&2 l G^4_{lj} - (l-1)^2 F^4_{l-1,j } - \frac{F^4_{l-1,j+1} - 2 F^4_{l-1,j }+F^4_{l-1,j-1}}{\Delta \theta^2}
= -\frac{\Delta \theta^2}{12}D_{4\theta}^2 F^2_{l-1,j},\label{RecurrDiscr1}\\
&2 l F^4_{lj} + l^2 G^4_{l-1,j } + \frac{G^4_{l-1,j+1} - 2 G^4_{l-1,j }+G^4_{l-1,j-1}}{\Delta \theta^2}
= \frac{\Delta \theta^2}{12}D_{4\theta}^2 G^2_{l-1,j},\label{RecurrDiscr2}
&\end{align}
where $F^2_{l-1,j}$ and $G^2_{l-1,j}$, obtained in the first step, are second order approximations of $F_{l-1,j}$ and $G_{l-1,j}$ which are part of the exact solution
of the original scattering BVP. Also, the discrete differential operator $D_{4\theta}^2 $ is defined by equation (\ref{D4theta}) of the previous section.

\begin{Remark}
The proofs that the finite difference formulas (\ref{BC44NH})-(\ref{BC54NH}) and (\ref{RecurrDiscr1})-(\ref{RecurrDiscr2}) approximate their continuous counterparts to fourth order are very similar to the proof of Theorem \ref{Th1}. Therefore, they are omitted. The key assumption for these proof is  that the discrete functions $U^2_{N,j}$,  $F^2_{l-1,j}$ and $G^2_{l-1,j}$,
which are obtained in step 1, are second order approximations of $u(R,\theta)$, $F_{l-1} (\theta)$, and 
$G_{l-1} (\theta)$, respectively.
\end{Remark}

Summarizing, the set of algebraic equations (\ref{4thorder}), (\ref{BC44NH})-(\ref{BC54NH}), (\ref{RecurrDiscr1})-(\ref{RecurrDiscr2}), the discrete versions of the continuity of the scattered field (\ref{BVPBd3}) and the boundary condition at the obstacle (\ref{BVPBd2}) form the fourth order Karp DC system of linear equations to be solved. We denote this system as {\it KDC4}. Details on the structure and solution of this linear system are given in Section \ref{Implementation}.

\subsection{General high order DC schemes for the Helmholtz equation in polar coordinates}\label{ssec:3.4}
The derivation of the fourth order DC scheme for the BVP (\ref{BVPBd1})-(\ref{Recurrence2}) modeled by the 
Helmholtz equation in subsections \ref{ssec:3.1}-\ref{ssec:3.2}, can be extended to obtain a scheme of arbitrary high order.
To start this derivation, 
we assume that $U^{(p-2)}_{ij}$ is a $(p-2)$th order discrete approximation
 of the solution $u$ of the Helmholtz equation (\ref{BVPBd1}) subject to the boundary conditions (\ref{BVPBd2})-(\ref{Recurrence2}). The details on the computation of this discrete approximation will be discussed later. 
Then, we apply  the standard centered second order finite difference approximation ${\mathcal{H}}_5^{2}$ of the Helmholtz operator to the solution $u$ and retain up to the $(p-2)$ leading order terms of the truncation errors of each Helmholtz derivative. As a result, we obtain
\begin{align}
{\mathcal{H}}^{2}_5u_{ij}&= \frac{u_{i+1,j} - 2 u_{ij} +u_{i-1,j}}{h^2} +
\frac{1}{r_i}\frac{u_{i+1,j} - u_{i-1,j}}{2\,h} 
+\frac{1}{r^2} \frac{u_{i,j+1} - 2 u_{ij} +u_{i,j-1}}{h^2} + k^2 u_{ij}=\nonumber\\ 
&- \left( \frac{1}{3!r_i}(u_{3r} )_{ij}+
\frac{2}{4!}(u_{4r})_{ij} + 
\frac{2}{4!r_i^2}(u_{4\theta})_{ij}\right)h^2 \nonumber \\
&- \left( \frac{1}{5!r_i}(u_{5r} )_{ij}+
\frac{2}{6!}(u_{6r})_{ij} + 
\frac{2}{6!r_i^2}(u_{6\theta})_{ij}\right)h^4  \ldots \nonumber \\
&- \left( \frac{1}{(p-1)!r_i}(u_{(p-1)r} )_{ij}+
\frac{2}{p!}(u_{pr})_{ij} + 
\frac{2}{p!r_i^2}(u_{p\theta})_{ij}\right) h^{p-2} + \mathcal{O}(h^p), \label{pscheme}
 \end{align}
where $p=4,6, \dots$, and $h\equiv\text{max}\{\Delta r, \Delta \theta\}$.

We continue the construction of the $p$th order DC approximation to the Helmholtz equation by imitating the one used to obtain the 4th order DC approximation (\ref{4thorder}). In fact, we proceed by subtracting all the error terms up to the  $(p-2)$th order from the middle member of the equation (\ref{pscheme}). 
This is followed by substituting the continuous derivatives present in these error terms, by 
appropriate finite difference approximations acting on the $(p-2)$th order discrete approximation, $U^{p-2}_{ij}$, of the exact solution $u$.
 More precisely, for $q=4,6,8,\dots,p$, we replace the continuous derivatives of the exact solution $\left(u_{qr}\right)_{ij}$, $\left(u_{(q-1)r}\right)_{ij}$, and $\left(u_{q\theta}\right)_{ij}$ in (\ref{pscheme}) by a
 $(p+2-q)$th order finite difference approximations given by 
 $D_{qr}^{p+2-q}U^{p-2}_{ij}$, $D_{(q-1)r}^{p+2-q}U^{p-2}_{ij}$, and $D_{q\theta}^{p+2-q} U^{p-2}_{ij}$, respectively. 
 This construction suggests the definition of 
 the following
$p$th order DC finite difference approximation to the Helmholtz differential operator,
\begin{align}
{\mathcal{H}}^{p}_5U^p_{ij}&\equiv \frac{U^p_{i+1,j} - 2 U^p_{ij} +U^p_{i-1,j}}{h^2} +
\frac{1}{r_i}\frac{U^p_{i+1,j} - U^p_{i-1,j}}{2\,h} 
+\frac{1}{r^2} \frac{U^p_{i,j+1} - 2 U^p_{ij} +U^p_{i,j-1}}{h^2} + k^2 U^p_{ij}\nonumber\\ 
&-\left( \frac{1}{3!r_i}D_{3r}^{p-2}U^{p-2}_{ij} +
\frac{2}{4!}D_{4r}^{p-2}U^{p-2}_{ij} + 
\frac{2}{4!r_i^2}D_{4\theta}^{p-2}U^{p-2}_{ij}\right)h^2 \nonumber \\
&-\left( \frac{1}{5!r_i}D_{5r}^{p-4}U^{p-2}_{ij} +
\frac{2}{6!}D_{6r}^{p-4}U^{p-2}_{ij} + 
\frac{2}{6!r_i^2}D_{6\theta}^{p-4}U^{p-2}_{ij}\right)h^4
- \ldots \nonumber \\
&-\left( \frac{1}{(p-1)!r_i}D_{(p-1)r}^{2}U^{p-2}_{ij} +
\frac{2}{p!}D_{pr}^{2}U^{p-2}_{ij} + 
\frac{2}{p!r_i^2}D_{p\theta}^{2}U^{p-2}_{ij}\right)h^{p-2}. \label{pthorder}
 \end{align}

We claim that the equation, ${\mathcal{H}}^{p}_5U^p_{ij}=0$, is consistent with the Helmholtz equation (\ref{BVPBd1}) of order  
$\mathcal{O}(h^{p})$. This is the content of our next theorem whose proof is provided below. The formulas for the discrete differential operators acting on $U^{p-2}_{ij}$, i.e., $D_{qr}^{p+2-q}U^{p-2}_{ij}$, $D_{(q-1)r}^{p+2-q}U^{p-2}_{ij}$, and $D_{q\theta}^{p+2-q} U^{p-2}_{ij}$ for arbitrary $p$ and $q=4,\dots p$, can be obtained by applying computational algorithms such as {\it fdcoeffF.m} written as a MATLAB function by Leveque \cite{LevequeBook}. As an illustrative example, we define the operators needed  to obtain a sixth order DC scheme, $p=6$ and $q=4,6$ in the \ref{AppendixA}.

Notice that the new finite difference operator $\mathcal{H}^p_5$ in (\ref{pthorder}) acts on the unknown discrete function ${U}^p_{ij}$, generating the same 5-point stencil of the standard centered second order discrete operator $\mathcal{H}^2_5$. The difference is in the additional known terms which depend on the numerical approximation $U^{p-2}_{ij}$ of the exact solution $u$. They are given by
\begin{align}
\left( \frac{1}{3!r_i}D_{3r}^{p-2}U^{p-2}_{ij} \right.&\left.+
\frac{2}{4!}D_{4r}^{p-2}U^{p-2}_{ij} + 
\frac{2}{4!r_i^2}D_{4\theta}^{p-2}U^{p-2}_{ij}\right)h^2
+\ldots \nonumber \\
&+\left( \frac{1}{(p-1)!r_i}D_{(p-1)r}^{2}U^{p-2}_{ij} +
\frac{2}{p!}D_{pr}^{2}U^{p-2}_{ij} + 
\frac{2}{p!r_i^2}D_{p\theta}^{2}U^{p-2}_{ij}\right)h^{p-2} \nonumber
\end{align}

\begin{Remark}
Near the artificial boundary $r = R$, we use appropriate one-sided finite difference to approximate the various derivatives present in the leading terms of the truncation error in (\ref{pthorder}).
\end{Remark}

\begin{Theorem}\label{Th2}
The new DC numerical scheme,
\begin{align}
{\mathcal{H}}^{p}_5U^{p}_{ij}\equiv {\mathcal H}^2_5 U^{p}_{ij} 
&-\left( \frac{1}{3!r_i}D_{3r}^{p-2}{U}^{p-2}_{ij} +
\frac{2}{4!}D_{4r}^{p-2}{U}^{p-2}_{ij} + 
\frac{2}{4!r_i^2}D_{4\theta}^{p-2}{U}^{p-2}_{ij}\right)h^2  \nonumber \\
&- \left( \frac{1}{5!r_i}D_{5r}^{p-4}{U}^{p-2}_{ij} +
\frac{2}{6!}D_{6r}^{p-4}{U}^{p-2}_{ij} + 
\frac{2}{6!r_i^2}D_{6\theta}^{p-4}{U}^{p-2}_{ij}\right)h^4
- \ldots \nonumber \\
&-\left( \frac{1}{(p-1)!r_i}D_{(p-1)r}^{2}{U}^{p-2}_{ij} +
\frac{2}{p!}D_{pr}^{2}{U}^{p-2}_{ij} + 
\frac{2}{p!r_i^2}D_{p\theta}^{2}{U}^{p-2}_{ij}\right)h^{p-2}=0,
\label{pthorderT}
 \end{align}
 is a consistent finite difference approximation of the Helmholtz equation in polar coordinates
$$  \Delta_{r\theta} u + k^2 u
= u_{rr} + \frac{1}{r} u_r + \frac{1}{r^2}u_{\theta\theta} +k^2 u=0, $$
of order $\mathcal{O}(h^{p})$ on $\Omega^{-}$, if the continuous function $u$ has derivatives of order $p$ in its two variables $r$ and $\theta$ on $\Omega^{-}$; and if
the discrete function $U^{p-2}_{ij}$ is a $(p-2)th$ order approximation of $u$, i.e.,
there exists $z(r,\theta)$ sufficiently smooth and bounded on the closure of $\Omega^{-}$ such that 
\begin{equation}
U^{p-2}_{ij} =u(r_i,\theta_j) +h^{p-2}z(r_i,\theta_j) \label{Up2}, \qquad\qquad\qquad\mbox{with}\quad h=\max \{\Delta r,\Delta \theta\}.
\end{equation}
\end{Theorem}

The following lemma is the analogue of Lemma \ref{Lemma1} for a $p$ ordered scheme.

\begin{Lemma}\label{Lemma2}
For $U^{p-2}_{ij}$ and $u(r,\theta)$ satisfying the hypothesis of Theorem \ref{Th2}, it holds that the standard $(p+2-q){\text{th}}$ order centered finite difference operator:
\begin{enumerate}[(i)]
\item  $D^{p+2-q}_{qr}$ acting on $U^{p-2}_{ij}$ is consistent with the derivative $u_{qr}$ of order $\mathcal{O}(h^{p+2-q})$, \\for $q=4,6\dots,p$.
\medskip

\item  $D^{p+2-q}_{(q-1)r}$ acting on $U^{p-2}_{ij}$ is consistent with the derivative $u_{(q-1)r}$ of order $\mathcal{O}(h^{p+2-q})$, \\for $q=4,6\dots,p$.

\medskip

\item  $D^{p+2-q}_{q\theta}$ acting on $U^{p-2}_{ij}$ is consistent with the derivative $u_{q\theta}$ of order $\mathcal{O}(h^{p+2-q})$, \\for $q=4,\dots,p$.
\end{enumerate}
 \end{Lemma}

\begin{proof}
To prove (i), we apply  $D^{p+2-q}_{qr}$ to $U^{p-2}_{ij}$ and use (\ref{Up2}). This leads to
\begin{align*}
D^{p+2-q}_{qr}U^{p-2}_{ij} &= 
D^{p+2-q}_{qr} u_{ij} + h^{p-2} D^{p+2-q}_{qr} z_{ij} \\
&= (u_{qr})_{ij} + \mathcal{O}(h^{p+2-q}) + h^{p-2} (z_{qr})_{ij} +  \mathcal{O}(h^{2p-q})\nonumber
\end{align*}
Therefore, 
\begin{align}
D^{p+2-q}_{qr}U^{p-2}_{ij} 
= (u_{qr})_{ij} + \mathcal{O}(h^{p+2-q}),
\end{align}
for $q=4,6\dots,p$. The proofs of parts (ii) and (iii) follow the same pattern.
\end{proof}

Now, we are ready to prove Theorem \ref{Th2}

\begin{proof} (Proof of Theorem \ref{Th2})\\
First, we rewrite $\mathcal{H}^{p}_5 { U}_{ij}$ as 
\begin{align}
 \mathcal{H}^p_5 { U}_{ij} = &
 D_{2r}^2 { U}_{ij} -\frac{2h^2}{4!}D_{4r}^{p-2}U^{p-2}_{ij} - \frac{2h^4}{6!}D_{6r}^{p-4}U^{p-2}_{ij} - \ldots -\frac{2h^{p-2}}{p!}D_{pr}^2U^{p-2}_{ij} \nonumber\\
 &+  \frac{1}{r_i}\left(D_{r}^2 {U}_{ij} -\frac{h^2}{3!}D_{3r}^{p-2}U^{p-2}_{ij} - \frac{h^4}{5!}D_{5r}^{p-4}U^{p-2}_{ij} - \ldots - \frac{h^{p-2}}{(p-1)!}D_{(p-1)r}^2U^{p-2}_{ij}\right) \nonumber\\
 &+ \frac{1}{r_i^2}\left(D_{2\theta}^2 {U}_{ij} 
- \frac{2h^2}{4!}D_{4\theta}^{p-2}U^{p-2}_{ij} - \frac{2h^4}{6!}D_{6\theta}^{p-4}U^{p-2}_{ij}-\ldots - \frac{2h^{p-2}}{p!}D_{p\theta}^2U^{p-2}_{ij}\right) + k^2 { U}_{ij}.\nonumber
 \end{align}
  Then applying $\mathcal{H}^p_5$ to $u$, which satisfy (\ref{Up2}),
 and expanding the individual terms where the discrete operators act on $u_{ij}$ results in
\begin{align*}
\mathcal{H}^{p}_5 u_{ij} &= 
(u_{rr})_{ij} + \frac{2h^2}{4!}\left(u_{4r}\right)_{ij} + \ldots + \frac{2h^{p-2}}{p!}(u_{pr})_{ij} + {\mathcal O}(h^p) \\
&\quad  -\frac{2h^2}{4!}D_{4r}^{p-2}U^{p-2}_{ij} - \ldots -\frac{2h^{p-2}}{p!}D_{pr}^2U^{p-2}_{ij} 
\\
&\quad+ \frac{1}{r_i}\left((u_{r})_{ij} + \frac{h^2}{3!}\left(u_{3r}\right)_{ij} + \ldots +  \frac{h^{p-2}}{(p-1)!}\left(u_{(p-1)r}\right)_{ij} +{\mathcal O}(h^p)\right. \\
&\quad \left. - \frac{h^2}{3!}D_{3r}^{p-2}U^{p-2}_{ij} - \ldots
-  \frac{h^{p-2}}{(p-1)!}D_{(p-1)r}^2U^{p-2}_{ij}\right)\\
&\quad+  \frac{1}{r_i^2}\left( (u_{\theta\theta})_{ij} + \frac{2h^2}{4!}\left(u_{4\theta}\right)_{ij} + \ldots + \frac{2h^{p-2}}{p!}(u_{p\theta})_{ij} + {\mathcal O}(h^p)\right. \\
&\quad \left. -\frac{2h^2}{4!}D_{4\theta}^{p-2}U^{p-2}_{ij} - \ldots -\frac{2h^{p-2}}{p!}D_{p\theta}^2U^{p-2}_{ij} \right) 
+ k^2 { u}_{ij}.
\end{align*}
Reordering and appropriately combining terms yields 
\begin{align*}
 \mathcal{H}^p_5 u_{ij}& =(u_{rr})_{ij} + \frac{1}{r_i}(u_{r})_{ij} +\frac{1}{r_i^2} (u_{\theta\theta})_{ij} +k^2 u_{ij} \\
  &- \frac{2h^2}{4!} \left(  D_{4r}^{p-2}U^{p-2}_{ij}  -   \left(u_{4r}\right)_{ij} \right) - \ldots
  - \frac{2h^{p-2}}{p!}\left(D_{pr}^2U^{p-2}_{ij} - (u_{pr})_{ij} \right) \\
  &- \frac{h^2}{3!r_i} \left(  D_{3r}^{p-2}U^{p-2}_{ij}  -   \left(u_{3r}\right)_{ij} \right)  - \ldots
     - \frac{h^{p-2}}{(p-1)!r_i} \left(  D_{(p-1)r}^2U^{p-2}_{ij}  -   \left(u_{(p-1)r}\right)_{ij} \right) \\
  &- \frac{2h^2}{4!r_i^2} \left(  D_{4\theta}^{p-2}U^{p-2}_{ij}  -   \left(u_{4\theta}\right)_{ij} \right)  - \ldots
  - \frac{2h^{p-2}}{p!r_i^2}\left(D_{p\theta}^2U^{p-2}_{ij} - (u_{p\theta})_{ij} \right) + {\mathcal O}(h^p).
\end{align*}
Thus, applying the statements of Lemma \ref{Lemma2} to each of the expressions in parentheses yields

\begin{align*}
 \mathcal{H}^p_5 u_{ij} = 
  \left(\Delta_{r\theta} u + k^2 u\right)_{ij}  + \mathcal{O}(h^p),
  \end{align*}
which finishes the proof.
\end{proof}

\subsection{Arbitrary order DC approximation for the KFE.}\label{ssec:3.5}
Following the derivation described in the previous three sections, we can obtain arbitrary order DC approximations for the KFE at the grid points $(r_N,\theta_j)$ by using appropriate one-sided finite difference for the derivatives present in the truncation error terms. For instance, the definitions of the $p$th order approximations for  
(\ref{BVPBd4}) and (\ref{BVPBd5}) are natural extensions of (\ref{BC44NH})-(\ref{BC54NH}). They consists of adding discrete approximations up to the $p$th order to all the continuous derivatives present in the truncation error terms . In fact,
\begin{align}
\frac{{U}^p_{N+1,j} - { U}^p_{N-1,j}}{2\Delta r} &-\,\,\partial_{r} \left( H_0(kr)\sum_{l=0}^{L-1} \frac{{ F}^p_{lj}}{r^l} + H_1(kr)\sum_{l=0}^{L-1} \frac{{ G}^p_{lj}}{r^l}\right)_{r=r_N=R}\label{BC4pNH} \\
&=\frac{h^2}{3!}(Dl)_{3r}^{p-2}U^{p-2}_{ij} +
\frac{h^4}{5!}(Dl)_{5r}^{p-4}U^{p-2}_{ij} + \ldots +
\frac{h^{p-2}}{(p-1)!}(Dl)_{(p-1)r}^2U^{p-2}_{ij} \nonumber\\
\frac{{U}^p_{N+1,j} - 2{ U}^p_{Nj}+{ U}^p_{N-1,j}}{\Delta r^2} &- \,\,\partial_{r}^2\left( H_0(kr)\sum_{l=0}^{L-1} \frac{{ F}^p_{lj}}{r^l} + H_1(kr)\sum_{l=0}^{L-1} \frac{{G}^p_{lj}}{r^l}\right)_{r=r_N=R}\label{BC5pNH}\\
&=\frac{2h^2}{4!}(Dl)_{4r}^{p-2}U^{p-2}_{ij} +
\frac{2h^4}{6!}(Dl)_{6r}^{p-4}U^{p-2}_{ij} + \ldots +
\frac{2h^{p-2}}{p!}(Dl)_{pr}^2U^{p-2}_{ij}. \nonumber
\end{align}
In the discrete equations (\ref{BC4pNH})-(\ref{BC5pNH}), we employ $(p+2-q)$th order discrete finite difference operators, $(Dl)_{(q-1)r}^{p+2-q}$ and $(Dl)_{qr}^{p+2-q}$ (for $q=4,6,\cdots,p$) acting on the discrete function $U^{p-2}_{ij}$ approximating $u$ of order ($p-2$)th. The use of $Dl$ instead of $D$ states that left one-sided finite difference approximations of the corresponding continuous derivatives are used.

Analogously, we formulate $p$th order approximations of the recurrence formulas as
\begin{align}
2 l {G}^p_{lj} - (l-1)^2 {F}^p_{l-1,j } &- \frac{{ F}^p_{l-1,j+1} - 2 { F}^p_{l-1,j }+{ F}^p_{l-1,j-1}}{\Delta \theta^2}
\label{RecurrDiscr1p}\\
&=-\frac{2h^2}{4!}D_{4\theta}^{p-2}F^{p-2}_{l-1,j} -
\frac{2h^4}{6!}D_{6\theta}^{p-4}F^{p-2}_{l-1,j} - \ldots -
\frac{2h^{p-2}}{p!}D_{p\theta}^2F^{p-2}_{l-1,j},\nonumber \\
2 l { F}^p_{lj} + l^2 { G}^p_{l-1,j } &+ \frac{{G}^p_{l-1,j+1} - 2 { G}^p_{l-1,j }+{ G}^p_{l-1,j-1}}{\Delta \theta^2}\label{RecurrDiscr2p}\\
&= \frac{2h^2}{4!}D_{4\theta}^{p-2}G^{p-2}_{l-1,j} +
\frac{2h^4}{6!}D_{6\theta}^{p-4}G^{p-2}_{l-1,j} + \ldots +
\frac{2h^{p-2}}{p!}D_{p\theta}^2G^{p-2}_{l-1,j},\nonumber
\end{align}
where $F^{p-2}_{l-1,j}$ and $G^{p-2}_{l-1,j}$ are part of the previously calculated ($p-2$)th ordered numerical solution, and the discrete operators $D_{q\theta}^{p+2-q}$ are centered finite difference operators. Summarizing, the set of equations (\ref{pthorder}), (\ref{BC4pNH})-(\ref{BC5pNH}), (\ref{RecurrDiscr1p})-(\ref{RecurrDiscr2p}), the discrete version of the continuity of the scattered field (\ref{BVPBd3}), and the appropriate discretization of the boundary condition at the obstacle (\ref{BVPBd2}) form the $p$th order DC discrete system of equations to be solved. We denote this system as {\it KDCp}.


\section{Standard fourth order numerical method for the KFE-BVP}  \label{sec:4}
In this section, we formulate a standard fourth order numerical method for the KFE-BVP (\ref{BVPBd1})-(\ref{Recurrence2}) in polar coordinates.  
This constitutes an 
alternative high order method for this BVP. In Section \ref {Section.Numerics2D}, we compare it with the DC fourth order method and access convergence, accuracy and computational efficiency of both. This fourth order method is also a natural extension of the standard second order finite difference method, carefully constructed in \cite{JCP2017}, where the KFE  was first introduced. 

We begin by considering the centered 
9-point standard finite difference scheme for the Helmholtz equation in polar coordinates, 
\begin{align}
& {\mathcal{H}}_9 {\bar U}^4_{ij} \equiv \frac{-{\bar U}^4_{i+2,j} + 16 {\bar U}^4_{i+1,j} - 30 {\bar U}^4_{i,j} +
  16 {\bar U}^4_{i-1,j} - {\bar U}^4_{i-2,j}}{12 \Delta r^2} +
  \frac{1}{r_i}\left(\frac{-{\bar U}^4_{i+2,j} + 8 {\bar U}^4_{i+1,j} - 8 {\bar U}^4_{i-1,j} + {\bar U}^4_{i-2,j}}{12 \Delta r}\right)        \nonumber \\
&\qquad + \frac{1}{r_i^2}\left(\frac{-{\bar U}^4_{i,j+2} + 16{\bar U}^4_{i,j+1} - 30{\bar U}^4_{i,j} + 16{\bar U}^4_{i,j-1} - {\bar U}^4_{i,j-2}}{12 \Delta \theta^2}\right)+ k^2 {\bar U}^4_{ij}=0. \label{H4thorder}
 \end{align}
at the interior gridlines $r_2 < r_i < r_{N-1}$. We adopt the alternative notation for the standard fourth order numerical solution $\bar{U}^4_{ij}$, because it is different from its DC counterpart $U^4_{ij}$, as confirmed by our numerical results in the next section.
At the boundaries of the domain, appropriate fourth order one-sided finite difference approximations of the various derivatives of the Helmholtz equation are required.

We also need non-centered fourth order finite difference approximations for the various equations of the KFE. For instance,
\begin{enumerate}[i)]
\item Continuity of the first derivative at the artificial boundary:
\begin{align}\label{Stand4drCont1Der}
&\frac{1}{\Delta r}\left[\frac{1}{4} \bar{U}^4_{N+1,j} +\frac{5}{6} \bar{U}^4_{N,j} - \frac{3}{2}\bar{U}^4_{N-1,j} + \frac{1}{2} \bar{U}^4_{N-2,j} -\frac{1}{12} \bar{U}^4_{N-3,j}
\right]\\
&\qquad\qquad\qquad\qquad\qquad\qquad\qquad-\,\partial_{r} \left( H_0(kr)\sum_{l=0}^{L-1} \frac{{\bar F}^4_{lj}}{r^l} + H_1(kr)\sum_{l=0}^{L-1} \frac{{\bar G}^4_{lj}}{r^l}\right)_{r=r_N} = 0. \nonumber 
\end{align}

\item Continuity of the second derivative at the artificial boundary:
\begin{align}\label{Stand4drCont2Der}
&\frac{1}{\Delta r^2}\left[\frac{5}{6} \bar{U}^4_{N+1,j} -\frac{5}{4} \bar{U}^4_{N,j} - \frac{1}{3} \bar{U}^4_{N-1,j} + \frac{7}{6} \bar{U}^4_{N-2,j} -\frac{1}{2} \bar{U}^4_{N-3,j} +\frac{1}{12} \bar{U}^4_{N-4,j}
\right]\\
&\qquad\qquad\qquad\qquad\qquad\qquad\qquad- \,\,\partial_{r}^2\left( H_0(kr)\sum_{l=0}^{L-1} \frac{{\bar F}^4_{lj}}{r^l} + H_1(kr)\sum_{l=0}^{L-1} \frac{{\bar G}^4_{lj}}{r^l}\right)_{r=r_N} = 0. \nonumber 
\end{align}

\item Standard fourth order discretization of the recursion formulas along the angular direction: 
\begin{align}
&2 l {\bar G}^4_{lj} - (l-1)^2 {\bar F}^4_{l-1,j } - \frac{-{\bar F}^4_{l-1,j+2} + 16 {\bar F}^4_{l-1,j+1} - 30 {\bar F}^4_{l-1,j} +
  16 {\bar F}^4_{l-1,j-1} - {\bar F}^4_{l-1,j-2}}{12 \Delta r^2} = 0,\\
&2 l {\bar F}^4_{lj} + l^2 {\bar G}^4_{l-1,j } + \frac{-{\bar G}^4_{l-1,j+2} + 16 {\bar G}^4_{l-1,j+1} - 30 {\bar G}^4_{l-1,j} +
  16 {\bar G}^4_{l-1,j-1} - {\bar G}^4_{l-1,j-2}}{12 \Delta r^2} 
= 0,\label{RecurrenceKS4}
&\end{align}
\end{enumerate}
for $l=1,\dots L-1$. Again, we have adopted the alternative notation for the standard fourth order numerical solutions $\bar{F}^4$, and $\bar{G}^4$ to differentiate them from their DC fourth order counterparts ${F}^4$, and ${G}^4$, respectively.
Notice that we have retained the values of the unknown functions at the ghost points $(r_{N+1},\theta_j)$ in all the one-sided finite difference approximations. As a consequence another set of unknowns is added to the problem. However, they are eliminated considering an additional set of equations given by the discretization of the Helmholtz equation (\ref{H4thorder}) at the nodes located on the artificial boundary  $r=r_N$. 
Our numerical experiments suggests that this practice leads to a more accurate and stable numerical solutions than just using typical one-sided finite differences.
The set of equations (\ref{H4thorder})-(\ref{RecurrenceKS4}) and the discrete version of the continuity $u$ (\ref{BVPBd3}) at the artificial boundary form the standard 4th order method for the KFE-BVP  (\ref{BVPBd1})-(\ref{Recurrence2}). We will denote this method as {\it KS4}.
 
At the computational level, this fourth order standard method reduces to a new linear system of equations (LSE) given by 
\begin{equation}
A_4{\bar{\bf U}}^4={\bf b}. \label{LSES4}
\end{equation}
 This matrix $A_4$ has a greater number of nonzero entries than 
the matrix associated to the fourth order DC method described above.
 As a consequence, memory and computing costs increase for this standard formulation. In Section \ref{Section.Numerics2D}, we compare both fourth order techniques through some numerical experiments. 

\section{Implementation of the DC numerical method coupled with the KFE absorbing boundary condition}
\label{Implementation}
The practical advantage of the DC method coupled with appropriate discretizations of the KFE is that it leads to arbitrary high  order numerical approximations to the solution of scattering BVPs, such as (\ref{BVPBd1})-(\ref{Recurrence2}). Here, we choose an obstacle of circular shape to alleviate the transformation of the KFE-BVP into the ultimate linear system. In  Section \ref{Conclusions}, we discuss the scattering from arbitrarily shaped scatterers. 
Our strategy to generate a  $p$th order DC numerical approximation ${U}^p_{ij}$  to the exact solution ${u}$ of the BVP (\ref{BVPBd1})-(\ref{Recurrence2}) can be summarized by the following steps: 

\begin{enumerate}[i)]
\item \label{step1} Obtain a second order approximation ${ U}_{ij}^{2}$ to the exact solution $u$ using a standard second order finite difference technique for the Helmholtz equation and the BCs. This technique was adopted in \cite{JCP2017}. As shown there, the set of discrete equations employed to obtain ${ U}_{ij}^{2}$ can be recast into the LSE, 
\begin{equation}
A_2{\bf U}^2={\bf b}. \label{LSE2}
\end{equation}
 In particular,
 the vector ${\bf U}^2$ consists of the unknown discrete value approximations of the scattered field ${u}$, and the unknown angular coefficients of the Karp's expansion for a given grid. The vector 
$\bf b$ is assembled from the boundary data at the obstacle generated from the incident wave. More precisely, the unknown vector ${\bf U}^2$ is defined as
\begin{eqnarray}
&&{\bf U}^2 =
\Big[ \overbrace{U^2_{1,1} ...U^2_{1,m}}^{\text{at obstacle }} ~ ~
\overbrace{U^2_{2,1} ...U^2_{2,m}...
U^2_{N-1,1}...U^2_{N-1,m}}^{\text{at interior grid points}} ~ ~\nonumber\\
&&\qquad\qquad\qquad\overbrace{F^2_{0,1}...F^2_{0,m}\, G^2_{0,1}...G^2_{0,m}...\,F^2_{L-1,1}...F^2_{L-1,m}\,G^2_{L-1,1}...G^2_{L-1,m}}^{\text{at artificial boundary}}
 \Big]^{T}, \label{VectorU2}
\end{eqnarray}
and the vector {\bf b}, for a Dirichlet boundary condition on the obstacle, as 
\begin{equation}
{\bf b} =
\Big[ \overbrace{-(u_{inc})_{1,1} ... -(u_{inc})_{1,m}}^{\text{at obstacle }} ~ 
      \overbrace{0 ... 0 ... 0 ... 0}^{\text{at interior grid points}} ~ 
      \overbrace{0 ... 0 ... 0 ... 0}^{\text{at artificial boundary}}
 \Big]^{T}, \label{Vectorb_stand2}
\end{equation} 
As a result, the matrix $A_2$ dimension is $(N-1+2L)m\times(N-1+2L)m$,
where the first $m\times m$ block corresponds to the identity matrix. In the case of a Neumann BC, minor updates to the $m$ first equations should be made. In this case, $A_2$ is affected by the ghost-point based treatment of the discretization of the radial derivative in (\ref{BVPBd2}), and $\bf b$ depends on the boundary data $\partial_{r} \uinc$. See details in \ref{AppendixB}. 

The sparse structure of the matrix $A_2$ is studied in \cite{JCP2017}, where built-in MATLAB linear solvers were employed to obtain second order accurate solutions. In fact, the complex matrix $A_2$ is non-Hermitian. The discretization of the KFE and the ordering of the discrete unknowns in ${\bf U}^2$ lead to a highly asymmetric block structure of the lower rows of $A_2$. A typical structure of $A_2$ is shown in Fig. \ref{fig:matrixbands}. This particular matrix corresponds to the first experiment described in Section \ref{Section.Numerics2D} and illustrated in Fig. \ref{fig:Scattering}. As can be seen in Fig. \ref{fig:matrixbands}, the matrix $A_2$ consists of mainly five diagonals, which are obtained from the 5-point scheme used in the finite difference approximation of the Helmholtz equation. But, it also has a non-symmetric tail corresponding to the unknowns angular functions of the karp's expansion.

In \cite{JCP2017}, the LU MATLAB solvers were successfully used for an ample set of scattering problems whose discretization led to $A_2$ type matrices. In this work, we also employ the direct built-in MATLAB linear solvers for our two dimensional high order DC schemes. As can be seen in our numerical experiments in Section \ref{Section.Numerics2D}, we also obtained excellent high order approximations for very refined grids up to 60 PPW and up to a maximum of 12 terms in the Karp's fairfield expansion.

\item \label{step2} Construct a fourth order finite difference DC consistent scheme of the Helmholtz equation by subtracting the second order leading terms of the truncation errors
 from the second order standard scheme (\ref{2ndorder}). 
This leads to the desired 
fourth order finite difference deferred correction discrete equation (\ref{4thorder}) consistent with the Helmholtz equation. Likewise, obtain fourth order approximations to the 
boundary conditions  (\ref{BVPBd2})-(\ref{Recurrence2}). For this purpose, use appropriate approximations to the various continuous derivatives present in 
the leading order truncation error terms using the second order approximation $U_{ij}^2$ obtained in step (\ref{step1}).
This process is described in detail in sections \ref{ssec:3.1}-\ref{ssec:3.2}. The resulting linear system is given by
\begin{equation}
 A_2 {\bf U}^4 = {\bf b} + {\bf b}^4_{DC}({\bf U}^2).
 \label{LSE4}
\end{equation}
The unknown vector ${\bf U}^4$ is identical to ${\bf U}^2$ in (\ref{LSE2}), except in the replacement of the superscript number 2 by 4. 
Also, the dependence of the correction vector ${\bf b}^4_{DC}$ on the second order numerical solution has been made explicit. This new vector consists of all the corrections in  (\ref{4thorder}), (\ref{BC44NH}), (\ref{BC54NH}), (\ref{RecurrDiscr1}) and (\ref{RecurrDiscr2}). In the particular case of Dirichlet BC, it reads
\begin{eqnarray}
&&{\bf b}^4_{DC}({\bf U}^2) =
\Big[ \overbrace{0 \cdots 0}^{\text{at boundary }} ~ ~
\overbrace{ \frac{\Delta r^2}{12}\left(D_{4r}^2 U^2_{2,1} + 
\frac{2}{r_2}D_{3r}^2 U^2_{2,1}\right) +
\frac{\Delta \theta^2}{12r_2^2}D_{4\theta}^2 U^2_{2,1} \cdots}^{\text{at interior grid points}} ~ ~\nonumber\\
&&\qquad\overbrace{ \frac{\Delta r^2}{12}\left(D_{4r}^2 U^2_{N-1,m} + 
\frac{2}{r_{N-1}}D_{3r}^2 U^2_{N-1,m}\right) +
\frac{\Delta \theta^2}{12r_{N-1}^2}D_{4\theta}^2 U^2_{N-1,m}}^{\text{at interior grid points (continued)}} ~ ~\nonumber\\
&&\qquad\overbrace{\frac{\Delta r^2}{12}\left((Dl)_{4r}^2 U^2_{N,1} + 
\frac{2}{r_{N}}(Dl)_{3r}^2 U^2_{N,1}\right) +
\frac{\Delta \theta^2}{12r_N^2}(Dl)_{4\theta}^2 U^2_{N,1}-
\left( \frac{\Delta r}{3} + \frac{\Delta r^2}{6 r_N} \right)(Dl)_{3r}^2 U^2_{N,1}\cdots}^{\text{at artificial boundary}} ~ ~\nonumber\\ 
&&\qquad\overbrace{\frac{\Delta r^2}{12}\left((Dl)_{4r}^2 U^2_{N,m} + 
\frac{2}{r_{N}}(Dl)_{3r}^2 U^2_{N,m}\right) +
\frac{\Delta \theta^2}{12r_N^2}(Dl)_{4\theta}^2 U^2_{N,m}-
\left( \frac{\Delta r}{3} + \frac{\Delta r^2}{6 r_N} \right)(Dl)_{3r}^2 U^2_{N,m}}^{\text{at artificial boundary (continued)}} ~ ~\nonumber\\
&&\qquad\overbrace{\frac{\Delta r^2}{12}(Dl)_{4r}^2 U^2_{N,1}-\frac{\Delta r}{3}(Dl)_{3r}^2 U^2_{N,1}\cdots
\frac{\Delta r^2}{12}(Dl)_{4r}^2 U^2_{N,m}-\frac{\Delta r}{3}(Dl)_{3r}^2 U^2_{N,m}}^{\text{at artificial boundary (continued)}} ~ ~\nonumber\\
&&\qquad\overbrace{-\frac{\Delta \theta^2}{12}D_{4\theta}^2 F^2_{0,1}\cdots
-\frac{\Delta \theta^2}{12}D_{4\theta}^2 F^2_{0,m}\ 
\frac{\Delta \theta^2}{12}D_{4\theta}^2 G^2_{0,1}\cdots
\frac{\Delta \theta^2}{12}D_{4\theta}^2 G^2_{0,m} }^{\text{at artificial boundary (continued)}}\nonumber\cdots \\
&&\qquad\overbrace{-\frac{\Delta \theta^2}{12}D_{4\theta}^2 F^2_{L-1,1}\cdots
-\frac{\Delta \theta^2}{12}D_{4\theta}^2 F^2_{L-1,m}\ 
\frac{\Delta \theta^2}{12}D_{4\theta}^2 G^2_{L-1,1}\cdots
\frac{\Delta \theta^2}{12}D_{4\theta}^2 G^2_{L-1,m} }^{\text{at artificial boundary (continued)}}
 \Big]^{T}. \label{Vectorb4}
\end{eqnarray}
Alternatively, under Neumann conditions, the first $m$ components of this vector must account for the correction term in (\ref{BCneuDC4}, Appendix B) combined to those arising from the ghost-point boundary treatment. By solving the linear system (\ref{LSE4}), we obtain a fourth order numerical approximation, $U_{ij}^4$, of the solution $u$ of the original BVP. 
 An important aspect of this approach is that the matrix $A_2$ remains the same in both the second and fourth order computations. The only difference with respect to the LSE (\ref{LSE2}) occurs in the forcing term.

\item 
Continue the iterative construction process described in (\ref{step2}) until a desired  $p$th order approximation $U_{ij}^p$ of the exact solution $u$ is obtained  from the previous approximation $U_{ij}^{p-2}$. The associated linear system in this general case is given by
\begin{equation}
A_2 {\bf U}^p = {\bf b} + {\bf b}^p_{DC}({\bf U}^{p-2}).
\end{equation}
The components of the vector ${\bf b}^p_{DC}$ consists of the approximations of the high order correction terms  present in (\ref{pthorder}), (\ref{BC4pNH}), (\ref{BC5pNH}), (\ref{RecurrDiscr1p}) and (\ref{RecurrDiscr2p}). All of them are computed  from ${\bf U}^{p-2}$. The unknown vector ${\bf U}^{p}$ is identical to ${\bf U}^2$ in (\ref{LSE2}), by replacing 2 by $p$ in all its components. 
Therefore, to obtain a numerical solution approximating the exact solution two orders higher, it is required to solve an additional  linear system. However, all the linear systems to be solved using the DC technique have the same sparse matrix $A_2$. The difference is in the forcing terms as explained above. Compared to high order standard schemes, this DC implementation feature is advantageous regardless of the adopted LSE resolution strategy. In the case of direct solvers, $A_2$ factorization is performed once and then reused throughout the high order DC iteration. In the case of iterative solvers, all the linear systems to be solved as part of the DC iteration
use the same solving algorithm, under the same preconditioning and optimal implementation. Therefore, the resolution strategy requires a one-time computational investment over a rather simple matrix for every DC iteration step.

\end{enumerate}

In Section \ref{Section.Numerics2D}, we perform numerical experiments employing the discrete DC Helmholtz operators  $\mathcal{H}^4_5$  and  $\mathcal{H}^6_5$  of fourth and sixth order, respectively. As a result, we obtain numerical solutions  approximating the scattered field $u$ of the KFE-BVP (\ref{BVPBd1})-(\ref{Recurrence2}).  
By comparing these numerical solutions against the corresponding exact solutions  for circular scatterers, the expected fourth and sixth order of convergence, respectively, are achieved. For completeness, several of the sixth order finite difference formulas, used in this work, approximating continuous derivatives are given in the \ref{AppendixA}. These continuous derivatives 
are contained 
in the leading order truncation error terms of the finite difference approximations of the Helmholtz equation and the equations forming the KFE, respectively. Also, the deferred corrections approximations for the Neumann boundary condition are given in \ref{AppendixB}.

%

\section{Numerical Results} \label{Section.Numerics2D}

In this section, we discuss numerical results for the scattering of a time harmonic incident plane wave, $ u_{inc} = e^{-i\omega t} e ^{ikx}$ propagating in the direction of the positive x-axis, from a circular obstacle. We obtain numerical solutions for our proposed  deferred correction method coupled with the Karp's farfield expansion ABC, which is applied to the KFE-BVP defined by (\ref{BVPBd1})-(\ref{Recurrence2}). To access the high accuracy and high order of convergence of this technique, we compare our numerical results  to the exact solution of  the exterior BVP defined by (\ref{BVPsc1})-(\ref{BVPsc3}) for a circular obstacle of radius $r_0$ and wavenumber $k$. This exact solution for the scattered field $u$ and its corresponding {\it Farfield Pattern} (FFP), $P(\theta)$  are as follows:
\begin{enumerate}
\item Dirichlet boundary condition: $u=-u_{inc}$,
\begin{equation}
u (r,\theta) = - \sum_{n=0}^{\infty} \epsilon_n i^n \frac{J_n(kr_0)}{H_n^{(1)}(kr_0)} H_n^{(1)}(kr)\cos n\theta
\end{equation}
\begin{equation}
 P(\theta) = -\left(\frac{2}{k\pi}  \right)^{1/2}e^{-i\pi/4} \sum_{n=0}^{\infty} \epsilon_n \frac{J_n(kr_0)}{H_n^{(1)}(kr_0)} \cos n\theta, 
\end{equation}
\item Neumann boundary condition: $\partial_{n} u = - \partial_{n} \uinc,$
\begin{eqnarray}
u (r,\theta) = - \sum_{n=0}^{\infty} \epsilon_n i^n \frac{J_n'(kr_0)}{H_n'^{(1)}(kr_0)} H_n^{(1)}(kr)\cos n\theta
\end{eqnarray}
\begin{equation}
P(\theta) = -\left(\frac{2}{k\pi}  \right)^{1/2}e^{-i\pi/4} \sum_{n=0}^{\infty} \epsilon_n \frac{J_n'(kr_0)}{H_n'^{(1)}(kr_0)} \cos n\theta,
\end{equation}
\end{enumerate}
where $\epsilon_0=1$ and $\epsilon_n=2$ for $n\ge 1$

In Fig. \ref{fig:Scattering}, we show some numerical results obtained by applying the fourth order deferred correction technique KDC4, introduced in Sections \ref{ssec:3.1}-\ref{ssec:3.2}, to the KFE-BVP. 
The left graphs illustrate the amplitude of the total field, $u_{total}=u_{inc}+u$ for both boundary conditions.  The middle graphs shows the comparison of the numerical Farfield Patterns (defined below) against the exact ones. The rightmost graphs show the fourth order of convergence of the numerical solution to the exact solution. In both numerical experiments, the wavelength is $2\pi$, the number of Karp's expansion terms is 9, the outer radius is 3, and the number of gridpoints per wavelength to determine the order of convergence is $[20, 30, 40, 50, 60]$. The infinite series defining the exact solutions have been truncated to 60 terms for our calculations. The structure of the $A_2$ matrix for the Dirichlet BVP of Fig. \ref{fig:Scattering} with 20 gridpoints per wavelength is shown in Fig. \ref{fig:matrixbands}.

In the following subsections, we present numerous results for the application of the DC technique to the KFE-BVP for the two BCs under studied. Discussion of the order of convergence and accuracy of the DC techniques, along with comparisons against other high order techniques are included as well. 

\begin{figure}[!h]
\begin{subfigure}{0.37\textwidth}
\includegraphics[width=\linewidth, height=5.5cm ]{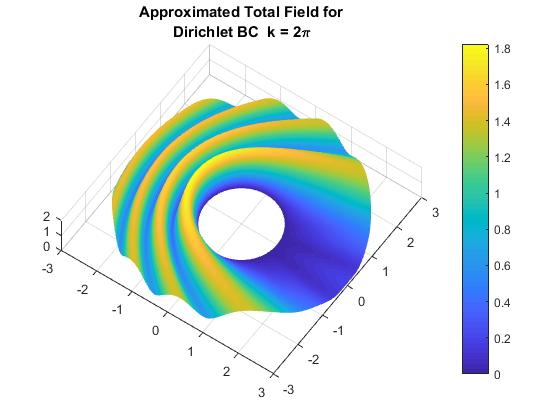} 
\end{subfigure}
\begin{subfigure}{0.37\textwidth}
\includegraphics[width=\linewidth, height=5.5cm]{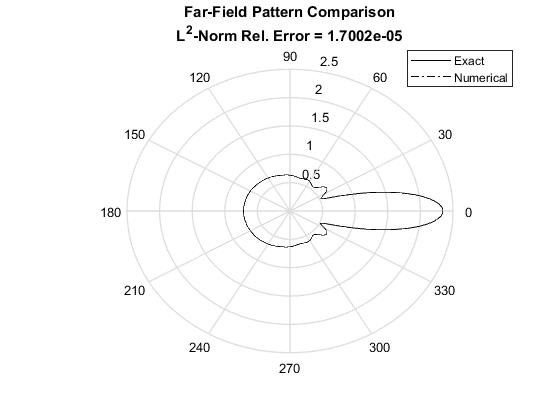}
\end{subfigure}
\begin{subfigure}{0.25\textwidth}
\includegraphics[width=\linewidth, height=4.cm ]{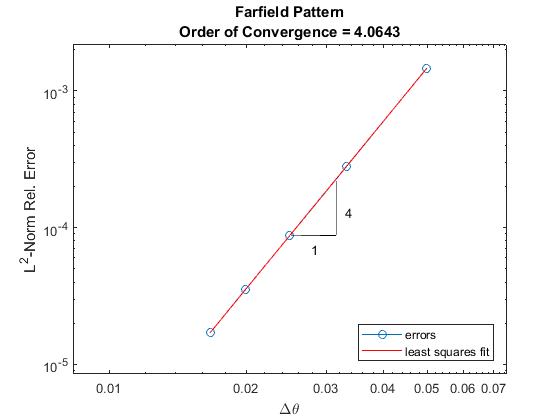}
\end{subfigure}
\begin{subfigure}{0.37\textwidth}
\includegraphics[width=\linewidth, height=5.5cm ]{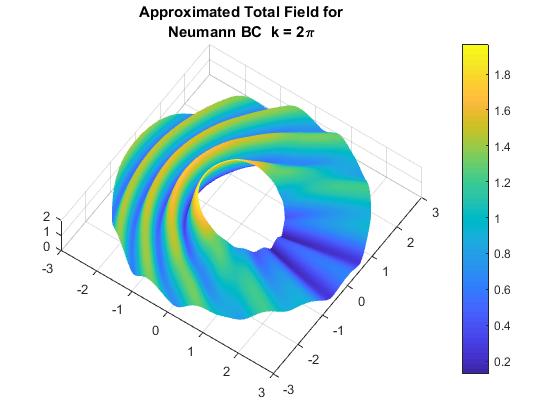} 
\end{subfigure}
\begin{subfigure}{0.37\textwidth}
\includegraphics[width=\linewidth, height=5.5cm]{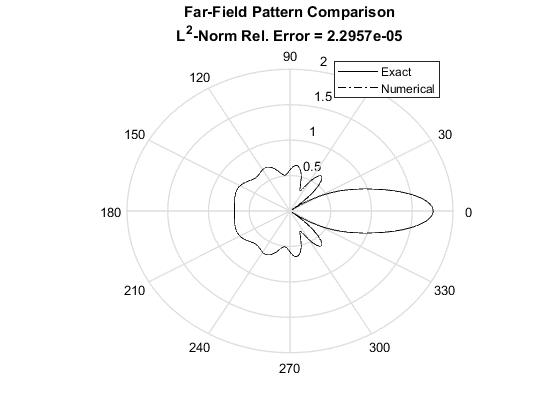}
\end{subfigure}
\begin{subfigure}{0.25\textwidth}
\includegraphics[width=\linewidth, height=4.cm ]{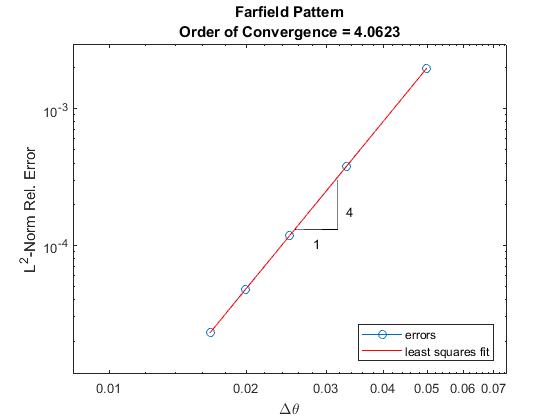}
\end{subfigure}
\vspace{-.7cm}
\caption{Numerical results for scattering from a circular scatterer using KFE. Shown from left to right are the wave amplitude, Farfield Pattern, and order of convergence for Dirichlet (top) and Neumann (bottom) BCs.} 
\label{fig:Scattering}
\end{figure}

\begin{figure}[h]
\begin{center}
\includegraphics[width=0.6\linewidth]{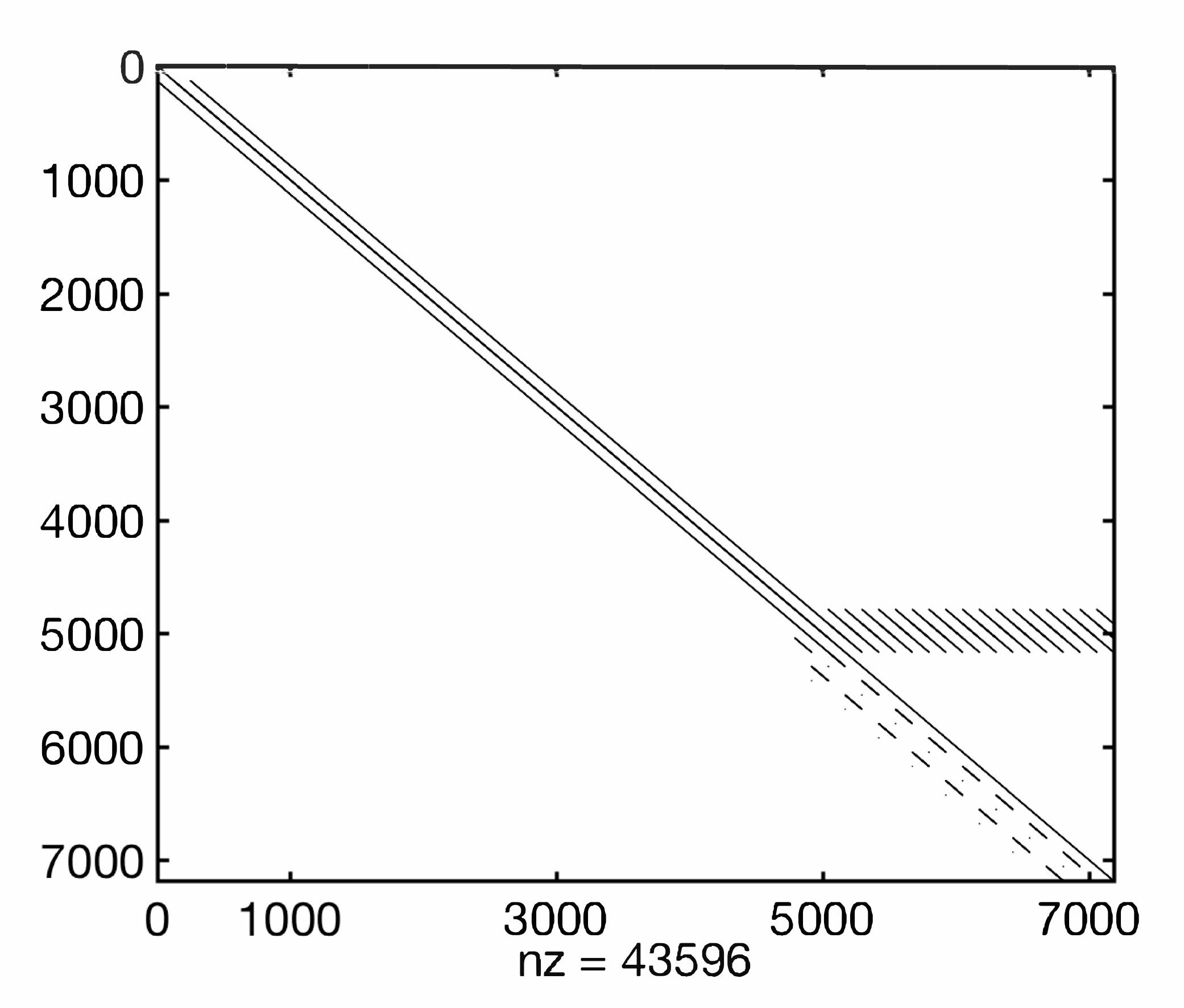} 
\end{center}
\vspace{-.7cm}
\caption{Matrix structure for Dirichlet BVP of  Figure \ref{fig:Scattering} with PPW=20 (grid size 40$\times$ 126) and NKFE=9.}
\label{fig:matrixbands}
\end{figure}

\subsection{Accuracy and convergence under Dirichlet BC} \label{ssec:6.1}

We perform several numerical experiments to obtain approximations of the FFP of the scattered wave for the BVP
(\ref{BVPBd1})-(\ref{Recurrence2}) under Dirichlet boundary conditions. 
The FFP is an important property to be analyzed in scattering problems. It depends on the shapes and physical properties of the scatterers.
In Section 4.2.1 of \cite{MartinBook}, Martin defines the Farfield Pattern (FFP) as the angular function present in the dominant term of the asymptotic expansions for  the scattered wave when $r\rightarrow\infty$. For instance, in 2D, the Farfield Pattern is the coefficient $f_0(\theta)$ of the leading order term of the asymptotic approximation of Karp expansion,
\begin{equation}
u(r,\theta)=
\frac{e^{ikr}}{(kr)^{1/2}}f_0(\theta)+O\left(1/(kr)^{3/2}\right). \label{KSFE}
\end{equation}
Following Bruno and Hyde \cite{McKaySIAM}, we  
calculate it from the  approximation of the scattered wave at 
the artificial boundary. A detailed computation is found in \cite{JCP2017}.

In our first set of experiments, we obtain the $L^2$ norm relative errors made by approximating the exact FFP by the numerical FFP using both the DC and standard techniques. 
In all these experiments,  
the wavenumber  is $k=2\pi$ and the radius of the circular obstacle is $r_0 =1$ while the radius of the artificial boundary is either $R=2$ or $R=3$. 
To determine the convergence rates of the numerical solutions to the exact solution,
 we refine our polar grid by increasing the number of gridpoints per wavelength (PPW) from 20 to 60.
These results are illustrated in Figs. \ref{fig:ConvergencePPW} and \ref{fig:OrderConv}. 
  In Fig. \ref{fig:ConvergencePPW}, we present four cases with varying number of Karp Expansion terms, NKFE = 4,8,10,13 while the grid is systematically refined. In these figures, we denote solutions using the DC methods of fourth and sixth order as KDC4 and KDC6, respectively. Also, those solutions obtained from the application of the standard second and fourth order schemes are referred as KS2 and KS4, respectively. 
  
 \begin{figure}[!ht]
\begin{subfigure}{0.5\textwidth}
\includegraphics[width=\textwidth]{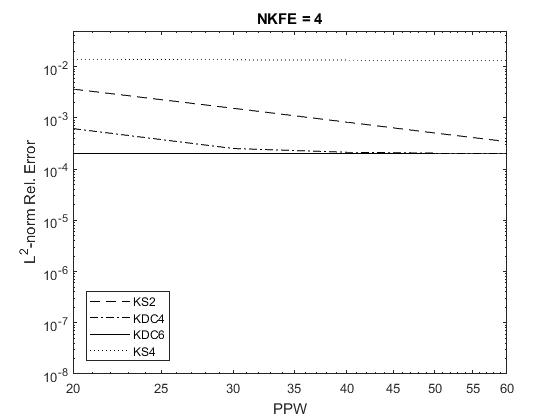} 
\end{subfigure}
\begin{subfigure}{0.5\textwidth}
\includegraphics[width=\textwidth]{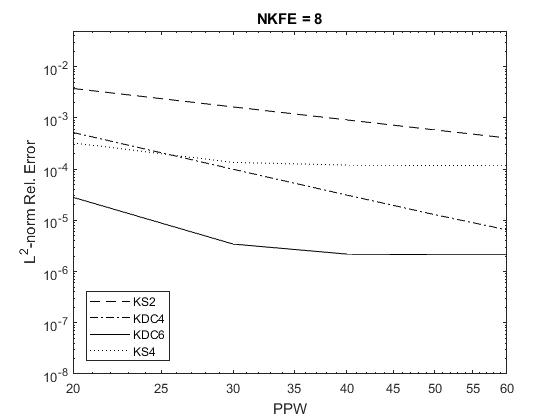}
\end{subfigure}
\begin{subfigure}{0.5\textwidth}
\includegraphics[width=\textwidth]{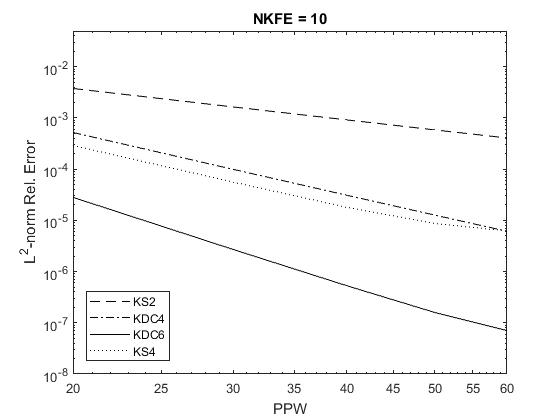}
\end{subfigure}
\begin{subfigure}{0.5\textwidth}
\includegraphics[width=\textwidth]{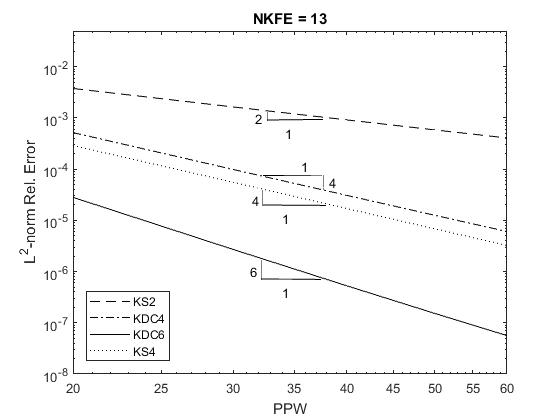}
\end{subfigure}
\caption{Comparison of L$^2$-norm relative errors of the Farfield Pattern computed from KS2, KDC4, KDC6 and KS4 methods for $R=2$} 
\label{fig:ConvergencePPW}
\end{figure}

\begin{figure}[h]
\begin{center}
\includegraphics[width=0.6\linewidth]{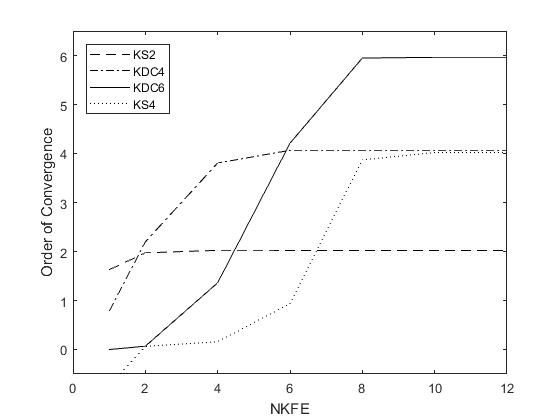} 
\end{center}
\caption{Comparison of convergence rates between the DC and standard methods in terms of NKFE. }
\label{fig:OrderConv}
\end{figure}

An analysis of the graphs in Fig. \ref{fig:ConvergencePPW} reveals how the accuracy and convergence rate of the standard and DC techniques, coupled with the KFE, depend on both PPW and the number of Karp's expansion terms. From the top left plot of Fig. \ref{fig:ConvergencePPW}, it is observed that the only method that attains the expected convergence rate (2nd order convergence), when NKFE =  4, is KS2.  
 Furthermore, KDC4 seems to attain the expected convergence rate only for coarser grids, while errors of the other methods remain nearly constant as PPW increases. As NKFE increases from 4 to 8 terms (top right plot of Fig. \ref{fig:ConvergencePPW}), we observe that KS2 continues attaining the second order rate, but now KDC4 also achieves the expected 4th order convergence rate along the full PPW range. However, KDC6 and KS4 convergence rates improve for the coarser grids, but as we continue refining, they degrade quickly.
 More consistent experimental results are attained by all methods for $10$ KFE terms, as depicted in the bottom left plot of Fig. \ref{fig:ConvergencePPW}, with the exception of KS4 whose order drops slightly below $4$ for the finer grids. Finally for NKFE = 13, all these methods attain their theoretical convergence rate, as shown in the bottom right plot of Fig. \ref{fig:ConvergencePPW}.

\begin{table}[h!]
\centering
\begin{tabular}{||c c c c c||}
\hline
   PPW    & Grid size & $h= r_{0} \Delta \theta = \Delta r$   & $L^2$-norm Rel. Error   & Observed order \\ [0.5ex]
   \hline\hline
$20$   & $40\times 126$  & $0.04987$  & $5.86\times 10^{-5}$  & $ $    \\
 $30$   & $60\times 189$  & $0.03324$  & $5.05\times 10^{-6}$  & $ 6.04$    \\
 $40$   & $80\times 252$  & $0.02493$  & $9.15\times 10^{-7}$  & $ 5.94$    \\
 $50$   & $100\times 315$  & $0.01995$  & $2.47\times 10^{-7}$  & $ 5.87$    \\
 $60$   & $120\times 377$  & $0.01667$  & $8.58\times 10^{-8}$  & $ 5.89$    \\
[1ex]
\hline
\end{tabular}
\caption{Order of convergence of the KDC6 method for Dirichlet BC when  k=2$\pi$, NKFE=13, and R=3.}
\label{table:6thorderR=3}
\end{table}

Alternatively, Fig. \ref{fig:OrderConv} shows the convergence rates attained by the DC and the standard methods as the number of terms in Karp expansion (NKFE) increases. In these experiments, the artificial boundary is located at $R=3$ and the various grids used to determine the convergence rate consist of PPW = 20, 30, 40, 50 and 60.
It is seen from Figs. \ref{fig:ConvergencePPW} and \ref{fig:OrderConv} that any of these four methods can reach its theoretical order of convergence, if enough terms in Karp's expansion are retained for sufficiently fine grids. 
This fact illuminates the arbitrary high order character of the KFE-ABC, whose accuracy is easily adjustable to the precision of the interior Helmholtz solver. In addition, the order of convergence of the proposed interior scheme can be efficiently increased by adding higher order error terms of the discretized Helmholtz differential operator into the DC numerical scheme. We present a detailed convergence process in Table \ref{table:6thorderR=3} for $k=2\pi$, NKFE=13 and R=3. This table clearly shows not only how the accuracy improves by refining the grid, but more important that the expected 6th order of convergence can be observed if enough terms are used in the Karp's expansion. 

The dependence of the L$^2$-norm FFP relative error on PPW and NKFE 
 is explored in Fig. \ref{fig:ErrorSurface} for the KDC6 method. Note that at coarser grids, increasing the number of Karp's expansion terms only leads to minor accuracy improvements. On the contrary, as the grids become more refined using larger NKFE, greater accuracy occurs. In fact, for NKFE $=8$, the L$^2$ relative error approaches $10^{-7}$ as the grid is refined. This represents a significant improvement over the results obtained using lower NKFE.
 
\begin{figure}[h]
\begin{center}
\includegraphics[width=1.2\linewidth]{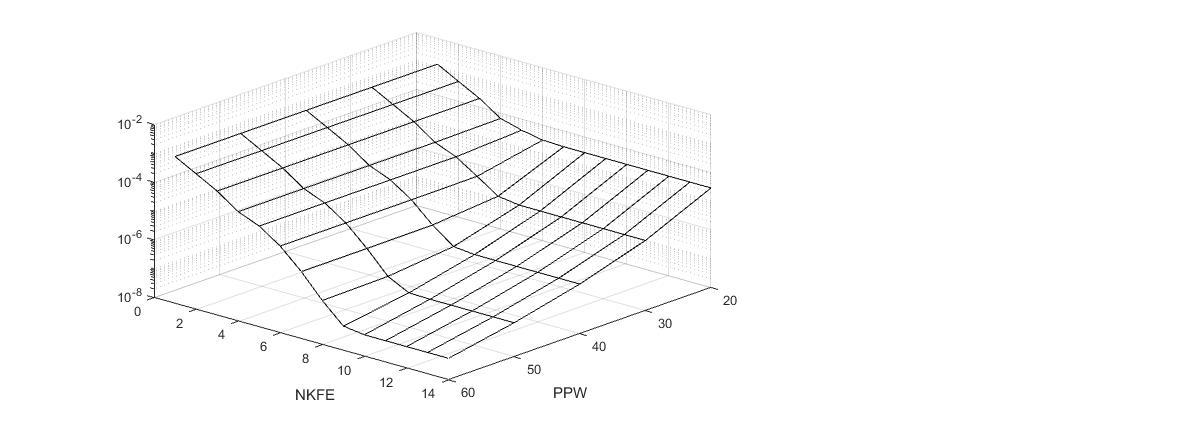} 
\end{center}
\caption{L$^2$-norm FFP relative error employing KDC6 method for various numbers of Karp's expansion terms and PPW values. Problem parameters are $r_0=1$, $k=2\pi$, $R=2$.}
\label{fig:ErrorSurface}
\end{figure}

\subsection{Comparison of computational times} \label{ssec:6.2}
We performed several experiments to evaluate the computational times spent by DC and Standard methods when solving similar problems to those in the previous section. The results of these experiments are 
depicted in Figure \ref{fig:time}. To obtain them, a set of baseline error tolerances are defined, and then pairs (NKFE,PPW) close to   optimal are found for each method to satisfy such thresholds. By an optimal pair we mean, the minimum values of NKFE and PPW that are needed to attain the target L$^2$-norm FFP relative error.
Once each pair (NKFE,PPW) is found, the same simulation is performed ten times. Then, the resulting CPU times are averaged and this time average is plotted in Fig. \ref{fig:time}.

For all simulations represented in Fig. \ref{fig:time}, the relevant parameters values are $r_0=1$, $k=2\pi$ and $R=3$. Also, the pairs needed (close to optimal) to attain the FFP target errors for the different methods are depicted in Table \ref {Table1}.
The curves on the left plot show that KS2 is not able to reach errors smaller than $10^{-4}$ during a time interval of 4 seconds, a threshold that is much earlier reached by the higher order methods, as depicted in the right plot.
In fact, it is observed that the time spent by KS2 to attain an error close to 10$^{-4}$ is about  8 times the one spent by KDC4. 
The limited KS2 accuracy prompted us to present two different plots with different time scales in Fig. \ref{fig:time}. 
We noticed that KDC6 is at least three times faster than both fourth order schemes, to reach an error of  $10^{-5}$. The KDC4 method was second best in terms of computational time. 
All these experiments were performed on an Intel Core i5 laptop computer of 8 GB RAM. The codes were written and executed in MATLAB R2017b. The results are not meant to represent our methods' limiting computational speed. Instead, these are presented to demonstrate the relative speed of the different methods described.

Notice, that we needed to choose $R=3$ to allow lower order methods to reach higher accuracies comparable to their higher order counterpart.

\begin{figure}[h!]
\begin{subfigure}{0.5\textwidth}
\includegraphics[width=\textwidth]{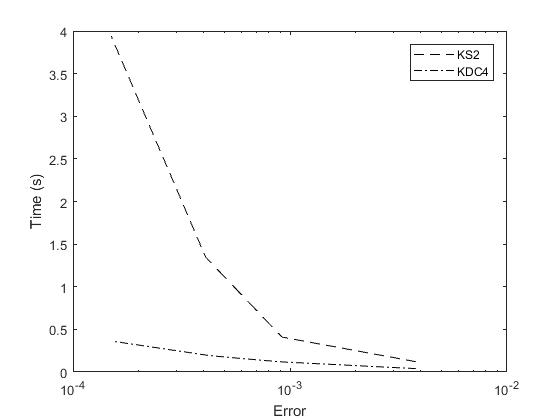} 
\end{subfigure}
\begin{subfigure}{0.5\textwidth}
\includegraphics[width=\textwidth]{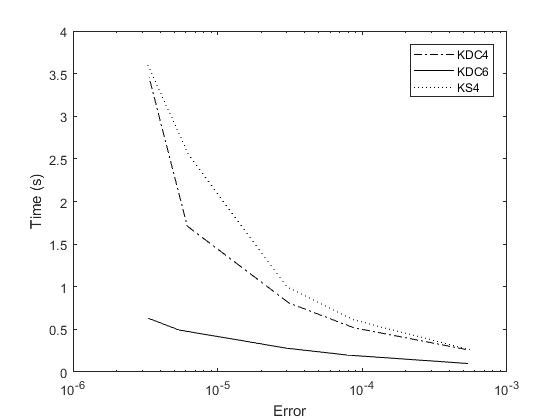}
\end{subfigure}
\caption{Computational times vs. L$^2$-norm FFP relative errors for the KS2, KDC4, KS4, and KDC6 methods.}
\label{fig:time}
\end{figure}

The sources of the time savings that are shown in Fig. \ref{fig:time} are most easily explained by examining Table \ref{Table1}. At all error levels, the DC6 scheme requires a much less refined grid and fewer Karp Expansion terms than DC4 or KS4. This leads to a smaller matrix and allows the scheme to be solved more efficiently despite the additional steps required. A similar phenomena can be seen when comparing DC4 and KS4 (especially in regards to the NKFE). The DC4 scheme have an additional efficiency advantage over the KS4 since the matrix used for solving the DC4 system is less dense than the matrix used by the KS4 method. As a note, the error values given in the table are approximations since it was not possible to find values of PPW and NKFE that would give identical errors for all three schemes. Additionally, the information displayed in Table \ref{Table1} corresponds to the right plot in Figure \ref{fig:time}, a similar table could be derived for the left plot.

\begin{table}[h!]
  \begin{center}
    \begin{tabular}{|c|c|c|c|c|c|c|}
    \hline
      \textbf{FFP Tol} & \multicolumn{2}{c|}{\textbf{KS4}} & \multicolumn{2}{c|}{\textbf{KDC4}} & \multicolumn{2}{c|}{\textbf{KDC6}}\\
      error & PPW & NKFE & PPW & NKFE &  PPW & NKFE\\
      \hline
      $\rule{0pt}{1.2em}5\times 10^{-4}$ & 17 & 8 & 20 & 8 & 13 & 2\\
      \hline
      $\rule{0pt}{1.2em}10^{-4}$ & 26 & 9 & 31 & 8 & 17 & 5\\
      \hline
      $\rule{0pt}{1.2em}5\times 10^{-5}$ & 35 & 9 & 40 & 8 & 20 & 6\\
      \hline
      $\rule{0pt}{1.2em}10^{-5}$ & 51 & 11 & 60 & 8 & 26 & 7\\
      \hline
      $\rule{0pt}{1.2em}5\times 10^{-6}$ & 60 & 11 & 70 & 8 & 29 & 8\\
       \hline
    \end{tabular}
    \caption{Points per wavelength  and KFE number of terms needed to reach a target FFP relative error.}
    \label{Table1}
  \end{center}
\end{table}

\subsection{Accuracy and convergence under Neumann BC} \label{ssec:6.3}
In this final section, we compare the FFP relative errors and convergence rates of fourth- and sixth-order DC solutions for the BVP (\ref{BVPBd1})-(\ref{Recurrence2}), under Neumann boundary conditions. For this boundary condition, the deferred correction schemes are detailed in the \ref{AppendixB}. In the following experiments, problem parameters, PPW, and NKFE, are the same ones used for the Dirichlet type tests in Section \ref{ssec:6.1}. In  Fig. \ref{fig:ConvergencePPWn}, we compare the FFP relative errors make by the application of KS2, KDC4 and KDC6 methods  for NKFE = $4, 8, 10$ and $13$, while 
in Fig. \ref{fig:OrderConvn}, we exhibit  the convergence rates dependence on NKFE. In both figures, the illustrated results nearly resemble the ones given in Figs. \ref{fig:ConvergencePPW} and \ref{fig:OrderConv}. Therefore, we conclude  that DC accuracy and convergence rates are not affected by the type of obstacle boundary condition. 

As in the Dirichlet case, the KDC4 and KDC6 methods require at least NKFE = 8 to fully reach their highest accuracy and expected convergence rates, respectively. Furthermore, for NKFE = 10 or 13, the KDC6 relative errors are two order of magnitude smaller than their KDC4 counterparts for sufficiently fine grids.
\begin{figure}[!ht]
\begin{subfigure}{0.5\textwidth}
\includegraphics[width=\textwidth]{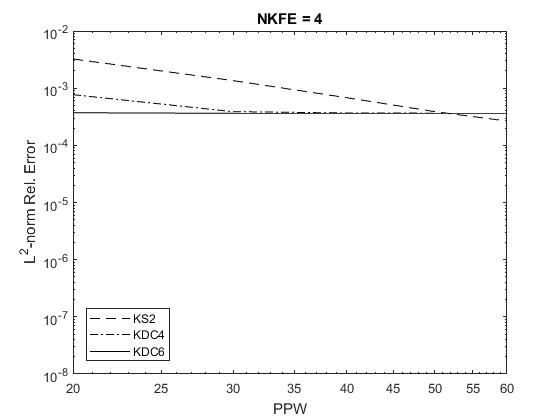} 
\end{subfigure}
\begin{subfigure}{0.5\textwidth}
\includegraphics[width=\textwidth]{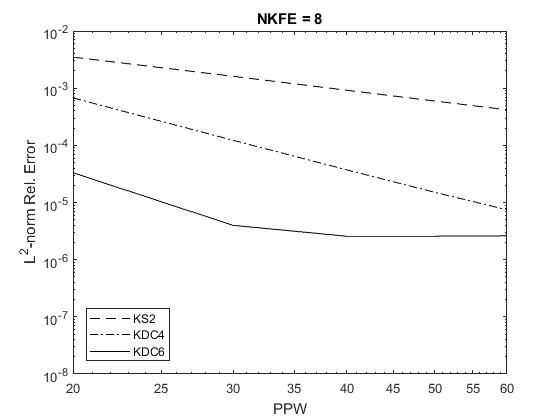}
\end{subfigure}
\begin{subfigure}{0.5\textwidth}
\includegraphics[width=\textwidth]{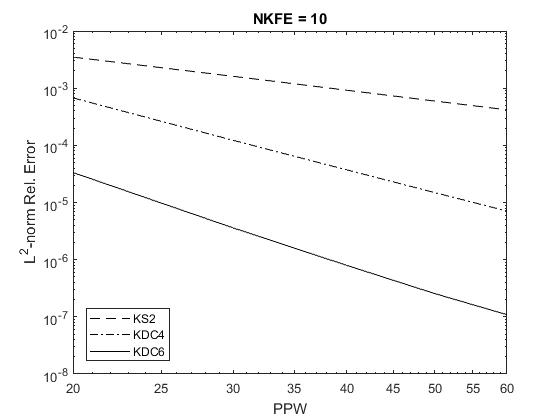}
\end{subfigure}
\begin{subfigure}{0.5\textwidth}
\includegraphics[width=\textwidth]{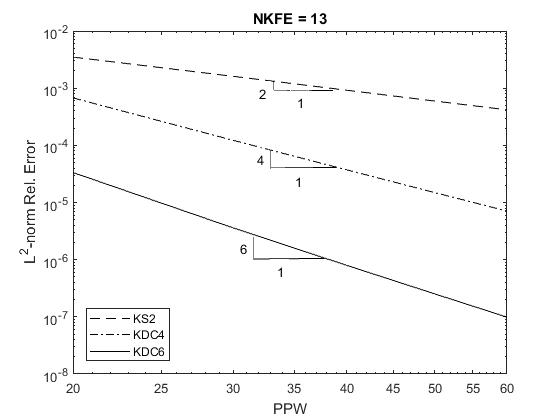}
\end{subfigure}
\caption{Comparison of L$_2$-norm FFP relative errors computed from KS2, KDC4 and KDC6 under Neumann BC.}
\label{fig:ConvergencePPWn}
\end{figure}

\begin{figure}[h]
\begin{center}
\includegraphics[width=0.6\linewidth]{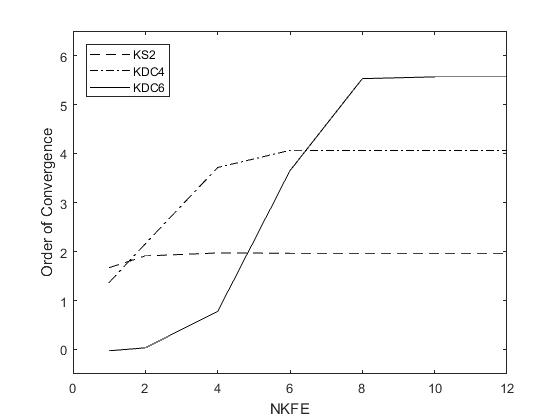} 
\end{center}
\caption{Comparison of convergence rates between the DC methods in terms of NKFE, under Neumann BC}
\label{fig:OrderConvn}
\end{figure}


\section{Concluding remarks and future work}
\label{Conclusions}

We have constructed an arbitrary high order numerical method for the two-dimensional time-harmonic acoustic scattering problem. This   is based on applying a general $p$th order DC technique to approximate the Helmholtz equation (interior approximation), and using an appropriate  number of terms (NKFE) for the the equations  (\ref{BVPBd3})-(\ref{Recurrence2}) defining the Karp's farfield expansion absorbing boundary condition. The KFE equations are also approximated by a $p$th order DC finite differences.

The DC approximation of the governing Helmholtz equation of arbitrary order $p$ is given by (\ref{pthorderT}), while the $p$th order DC approximation of the Karp's farfield expansion ABC with enough NKFE terms is given by the discrete equations  
 (\ref{BC4pNH})-(\ref{BC5pNH}), (\ref{RecurrDiscr1p})-(\ref{RecurrDiscr2p}). 
The algebraic linear system for the scattered field obtained from all these discretizations is completed  by
the discrete equations corresponding to the continuity of the scattered field (\ref{BVPBd3}) at the artificial boundary  and the appropriate discretization of the boundary condition at the obstacle (\ref{BVPBd2}). In the case case of Neumann or Robin condition, a
 $p$th order DC finite difference
discretization of the boundary condition at the obstacle should be constructed. 
The details in the construction of this DC scheme for the BVP  (\ref{BVPBd1})-(\ref{Recurrence2}) are given in Section \ref{ssec:3}. 

It is seen from Figs. \ref{fig:ConvergencePPW} - \ref{fig:ErrorSurface} and Figs. \ref{fig:ConvergencePPWn} - \ref{fig:OrderConvn} that the proposed method can reach its theoretical order of convergence, if enough terms in Karp's expansion are retained for sufficiently fine grids. This fact confirms the arbitrary high order character of the KFE-ABC that was claimed in \cite{JCP2017}. Actually, this high order property was already observed in \cite{CMAME2019} where the KFE-ABC was combined with a high order isogeometric finite element method employed as a Helmholtz solver for the interior. 

The virtue of this approach is that any  $p$th order approximation of the
Helmholtz equation consists of the same 5-point stencil which is used by the standard centered second order finite difference approximation.
 The difference between these two approximations is that the right hand side of the $p$th order scheme includes some additional terms that come from leading terms of the truncation error of the Helmholtz equation centered second order finite difference approximation.
  These new terms are calculated from a $(p-2)$th order numerical solution ${\bf U^{p-2}}$, previously computed. Hence, the proposed DC method is an iterative technique. For instance, the application of the DC fourth order method that leads to a fourth order approximation ${\bf U^{4}}$ is preceded by  the calculation of 
 ${\bf U^{2}}$ by applying a second order DC scheme.  Moreover, the matrix defining the LSE associated to both steps is the same matrix $A_2$. However, it is important to notice (as shown in Fig. \ref{fig:matrixbands})  that this matrix, although highly sparse, is not banded due to the presence of the unknown angular functions of the Karp's expansion.
  In Section \ref{Implementation}, we describe this iterative process in detail. 

In Section \ref{ssec:3}, we rigorously proved how our proposed DC finite difference method approximates the equations (\ref{BVPBd1})-(\ref{Recurrence2}), defining the KFE-BVP, to any order $p$. In Section \ref{Section.Numerics2D}, these theoretical results are experimentally confirmed by showing that the numerical solution obtained by applying KDC4 and KDC6,  indeed converge to the exact solution with a fourth and sixth order convergence rate, respectively. 

We accessed the computational effectiveness of our proposed technique by choosing target tolerance errors to be satisfied by the numerical FFP.  An analysis of the right plot in Fig. \ref{fig:time} reveals that KDC4 reaches the tolerance errors faster than its standard counterpart KS4. This is remarkable since the combined KDC4 technique requires the solution of two linear systems while KS4 needs to solve only one linear system. 
We attribute this performance to the greater sparsity of KDC4 matrix $A_2$ compared with the less sparse matrix $A_4$, associated to the  9-point standard finite difference approximation of Helmholtz equation. We also  notice that KS4 requires more KFE terms than KDC4 for similar grid sizes to reach a given tolerance error.  
But, more important KDC6 is eight times faster than KDC4, although its application has one more step in the iterative process. The reason for this is the coarser grid and lower number of KFE terms used by KDC6 compared with those employed by KDC4. For instance, KDC6 reaches an error close to $10^{-5}$ for (PPW,NKFE) = (26,7), while KDC4 needs (PPW,NKFE) = (60,8). 

In this work, we chose to limit our study to the two-dimensional Helmholtz equation in polar coordinates for clarity in the formulation and presentation of the theoretical results. However, we anticipate that an extension of the DC technique to the KFE-BVP in generalized curvilinear coordinates  will follow a similar pattern of what we have already observed in polar coordinates. We will base this extension on Villamizar and Acosta previous works on the grid generation for single and multiple scatterers of complexly shaped geometries \cite{ETNA2009,MATCOM2009}.  These grids correspond to generalized curvilinear coordinates conforming to the boundaries of the scatterers. Moreover, this grids were used by the same authors to solve single and multiple scattering problems  from complexly shaped obstacles in the following articles \cite{JCP2010,COMAME2012,JCP2017}. In particular in \cite{JCP2017}, they obtained second order convergence for a star shaped scatter with smooth boundary by using second order finite difference approximation based on curvilinear coordinantes conforming to the scatterer boundary. 

To provide some insight into the formulation of the DC finite difference technique in curvilinear coordinates, we consider 
 the Helmholtz equation in the curvilinear coordinates $(\xi,\eta)$, derived by Winslow,
\begin{equation}
\alpha u_{\xi \xi} - 2 \beta u_{\xi \eta} + \gamma u_{\eta \eta} + k^2 J^2 u = 0 \label{Winslow}
\end{equation}
where the coefficients $\alpha=\alpha(\xi,\eta)$, $\beta=\beta(\xi,\eta)$, $\gamma=\gamma(\xi,\eta)$ and $J=J(\xi,\eta)$ are defined by a coordinates transformation. If these coefficients are known exactly, then the deferred correction approach presented in this paper can be directly extended to (\ref{Winslow}) to approximate the derivatives $u_{\xi \xi}$, $u_{\xi \eta}$ and $u_{\eta \eta}$ up to a desired order. However, in practice the coefficients $\alpha$, $\beta$, $\gamma$ and $J$ are also approximated with finite difference schemes. Hence, in order to maintain the desired order of convergence, the coefficients $\alpha$, $\beta$, $\gamma$ and $J$ of the coordinates transformation must be approximated to match the same desired order for the truncation error. Our immediate plan is to attempt this extension to high order methods in curvilinear coordinates within the framework of DC methods. 

Another direction for future work will be to consider an extension of the DC technique for scattering problems in heterogeneous media (variable wavenumber) which can be handled by finite difference methods. In fact, the approach to be followed for the generalized curvilinear coordinates extension my also apply for this one. 

One more extension within our reach is the construction of the DC finite difference approach for the three-dimensional 
scattering modeled by the Helmholtz equation coupled with a high order local farfield expansions ABC. For this, we will start by considering 
the three-dimensional Helmholtz equation in spherical coordinates,
\begin{equation}
\Delta_{r,\theta,\phi}u + k^2 u = u_{rr} + \frac{2}{r} u_r + \frac{1}{r^2 \sin\theta}(\sin\theta\,\, u_{\theta})_{\theta} + \frac{1}{r^2\sin^2\theta}u_{\phi\phi} + k^2 u=0,\label{3D}
\end{equation}
coupled with the high order local farfield expansion ABC for the three-dimensional case (WFE), which was constructed in \cite{JCP2017} from the Wilcox's farfield expansion (\ref{Wilcox}). In this previous work,  it was shown that combining a standard centered second order finite difference method in the interior with the WFE at the artificial boundary, having enough terms, leads to a second order convergence rate of the numerical solution to the exact solution. Therefore, we have all the ingredients to derive a higher order DC finite difference method for the acoustic scattering problem WFE-BVP governed by (\ref{3D}). In fact, by approximating the derivatives present in (\ref{3D}) and in the WFE absorbing boundary condition, and keeping the necessary truncation errors terms, a DC approximation of any order can be derived. Nonetheless, these new DC developments may inevitably require iterative LSE resolution algorithms, suitable for large and sparse, non-Hermitian, and poorly conditioned (under large wavenumbers) matrices. The high demand of computer memory of direct solvers, make themselves infeasible choices in the case of 3-D refined grids. To downsize study efforts, we will consider some Krylov subspace methods for the Helmholtz equation, including e.g., the work by Kechroud et al.\cite{kechroud-2004}, Erlangga \cite{Erlangga-2008} and 
Gordon and Gordon \cite{Gordon-2013}.
At the discrete level, the local WFE transforms into a unique  sparse matrix structure, that highly motivates the exploration for efficient LSE solvers.
Furthermore, we want to emphasize that the DC procedure developed in this work is not limited to the BVP modeled by the Helmholtz equation. Indeed, it can be easily applied to any other BVP modeled by linear partial differential equations. 

 In its present form the KFE can only be applied to problems in the full-plane since its artificial boundaries need to be circles. However, we foresee that it can be adapted to wave problems in the half-plane with an acoustically soft or hard condition on the plane boundary. For this adaptation, we will use our recently developed multiple KFE absorbing boundary condition \cite{Waves2019}. In fact, we plan to use the method of images to extend the single scattering problem in the half-plane to a multiple scattering problem containing two identical scatterers in the full-plane. Then, we will use the multiple-KFE condition and symmetry relations between the outgoing waves to construct the KFE condition for the half-plane. Our purpose is to imitate the procedure employed by Acosta and Villamizar in \cite{COMAME2012} for the construction of the Dirichlet to Neumann (DtN) condition for a single obstacle in the half-plane from the multiple DtN condition for the full-plane \cite{Grote01,JCP2010}. However, a formulation of the KFE for waveguides or more arbitrary geometries may not be possible.

As parallel efforts, we are aware of several attempts to formulate high order methods for time-harmonic acoustic scattering using finite element techniques. A rather complete set of these contributions can be found in \cite{Dinachandra-Raju2018}. Most of them have been done without employing high order local ABC. 
However, we are aware of three works where high order finite element basis were used in combination with high order local ABCs to formulate overall high order methods for acoustic scattering. For instance, Schmidt and Heier \cite{schmidt2015} used high order Lagrange polynomials as a finite element basis combined with Feng's ABC of several orders. They obtained high order of convergence but only for wavenumber $k=1$. Moreover, Feng's absorbing boundary condition is obtained from an asymptotic expansion of the exact Dirichlet-to-Neumann ABC. This is a disadvantage compared with the Karp's expansion, which is an exact representation of the outgoing wave outside the artificial boundary. Also,
Barucq et. al \cite{Barucq2014} 
derived an ABC for exterior problems modeled by the Helmholtz equation in the plane from an approximation of the Dirichlet-to-Neumann map. They reached only fourth order convergence for circular obstacles.
More recently, Tahsin and Villamizar \cite{CMAME2019}, formulated an overall arbitrary high order method for acoustic scattering combining a finite element implementation in isogeometric analysis (IGA), which uses arbitrary high order
non-uniform rational B-splines (NURBS) as a basis, with the high order local farfield expansions ABC employed in this study. Their experiments corroborated the overall high order convergence for high and very low frequencies (wavenumbers). Our plan is to compare the accuracy and computational cost performances 
of our overall high order DC scheme combined with the KFE, constructed in this work, against the high order method introduced in \cite{CMAME2019} and  report these results elsewhere.

\appendix


\section{Sixth order DC finite difference approximations} \label{AppendixA}

The sixth order DC finite difference approximation to the Helmholtz differential operator is obtained by substituting $p=6$ in (\ref{pthorderT}) which leads to
\begin{align}
{\mathcal{H}}^{6}_5U^{6}_{ij}\equiv {\mathcal H}^2_5 U^{6}_{ij} 
&-\left( \frac{1}{3!r_i}D_{3r}^{4}{U}^{4}_{ij} +
\frac{2}{4!}D_{4r}^{4}{U}^{4}_{ij} + 
\frac{2}{4!r_i^2}D_{4\theta}^{4}{U}^{4}_{ij}\right)h^2\nonumber \\
&-\left( \frac{1}{5!r_i}D_{5r}^{2}{U}^{4}_{ij} +
\frac{2}{6!}D_{6r}^{2}{U}^{4}_{ij} + 
\frac{2}{6!r_i^2}D_{6\theta}^{2}{U}^{4}_{ij}\right)h^4,
\label{6thorder}
 \end{align}
where 
 $D_{4r}^4U^4_{ij}$, $D_{3r}^4U^4_{ij}$, and $D_{4\theta}^4U^4_{ij}$ are fourth order approximations of 
 $\left(u_{4r}\right)_{ij}$, $\left(u_{3r}\right)_{ij}$, and $\left(u_{4\theta}\right)_{ij}$, respectively. They are obtained
by applying standard centered fourth order finite difference to  $U^4_{ij}$ and are given by
\begin{align}
&\left(u_{4r}\right)_{ij}\approx D_{4r}^4{U}^4_{ij}
\equiv 
\frac{1}{\Delta r^4}\left[
\frac{-1}{6} {U}^4_{i-3,j}
 + 2{U}^4_{i-2,j} - \frac{13}{2} U^4_{i-1,j} + \frac{28}{3} U^4_{i,j}
-\frac{13}{2}U^4_{i+1,j}
+2 U^4_{i+2,j} -\frac{1}{6}U^4_{i+3,j}
\right]\label{D4r4}
\\
&\left(u_{3r}\right)_{ij}\approx D_{3r}^4U^4_{ij}
\equiv 
\frac{1}{\Delta r^3}
\left[
\frac{1}{8} U^4_{i-3,j}
 -U^4_{i-2,j} + \frac{13}{8} U^4_{i-1,j} -
\frac{13}{8}U^4_{i+1,j} + U^4_{i+2,j} -\frac{1}{8}U^4_{i+3,j}
\right]\label{D3r4}
\\
&\left(u_{4\theta}\right)_{ij}\approx D_{4\theta}^4U^4_{ij}
\equiv 
\frac{1}{\Delta \theta^4}\left[
\frac{-1}{6} U^4_{i,j-3}
 + 2U^4_{i,j-2} - \frac{13}{2} U^4_{i,j-1} + \frac{28}{3} U^4_{i,j}
-\frac{13}{2}U^4_{i,j+1}
+2 U^4_{i,j+2} -\frac{1}{6}U^4_{i,j+3}
\right].\label{D4theta4}
\end{align}

Also,
 $D_{6r}^2U^4_{ij}$, $D_{5r}^2U^4_{ij}$, and $D_{6\theta}^2U^4_{ij}$
 are second order approximations of 
 $\left(u_{6r}\right)_{ij}$, $\left(u_{5r}\right)_{ij}$, and $\left(u_{6\theta}\right)_{ij}$, respectively. They are obtained
by applying standard centered second order finite difference to  $U^4_{ij}$ and are given by
\begin{align}
&\left(u_{5r}\right)_{ij}\approx D_{5r}^2U^4_{ij}
\equiv 
\frac{1}{\Delta r^5}  \left[
\frac{-1}{2} U^4_{i-3,j}
 + 2U^4_{i-2,j} - \frac{5}{2} U^4_{i-1,j} 
+\frac{5}{2}U^4_{i+1,j}-2 U^4_{i+2,j} +\frac{1}{2}U^4_{i+3,j}
\right]\label{D5r2} \\\
&\left(u_{6r}\right)_{ij}\approx D_{6r}^2U^4_{ij}
\equiv 
\frac{1}{\Delta r^6}\left[
 U^4_{i-3,j}
 -6U^4_{i-2,j}  + 15U^4_{i-1,j} -20U^4_{i,j}
+15U^4_{i+1,j}
 -6 U^4_{i+2,j} +U^4_{i+3,j}\right]\label{D6r2}\\
&\left(u_{6\theta}\right)_{ij}\approx D_{6\theta}^2U^4_{ij}
\equiv 
\frac{1}{\Delta \theta^6}\left[
 U^4_{i,j-3}
-6U^4_{i,j-2}  + 15U^4_{i,j-1} -20U^4_{i,j}
+15U^4_{i,j+1} 
-6U^4_{i,j+2} +U^4_{i,j+3}.
\right]\label{D6theta}
\end{align}
For points close to the artificial boundary, appropriate one-sided finite difference approximations are employed.

\section{Approximations for the Neumann boundary condition} \label{AppendixB}

In the case of the Neumann condition at the boundary of a circular obstacle of radius $r_0$, the strategy to be employed for the construction of the fourth order approximation follows very closely the one employed at the artificial boundary for the KFE in subsection \ref{ssec:3.2}. 

First, we consider the standard second order centered finite difference approximation for the Neumann boundary condition at $r=r_0$, retaining an approximation of its leading order truncation error,
\begin{align}
\frac{U^4_{2,j}-U^4_{0,j}}{2\Delta r}  = -\frac{\partial \uinc}{\partial r} + \frac{\Delta r^2}{6} \left(Dr\right)_{3r}^2 U^2_{1,j}. \label{BCneuDC4} 
\end{align}
This equation contains the ghost values $U^4_{0,j}$. They are also present in the fourth order approximation of the Helmholtz equation (\ref{4thorder}) evaluated at $i=1$. Therefore, they can be eliminated by combining these equations. The equations (\ref{BCneuDC4}) and (\ref{4thorder}) also contain second order approximations of one-sided third and fourth forward radial derivatives at $r=r_0$, respectively. They act on the second order approximations $U^2_{1,j}$ of $u$ obtained in the first step and they are given by
\begin{align}
& \left(Dr\right)_{3r}^2U^2_{1,j} \equiv 
\frac{1}{\Delta r^3} \left[ -\frac{3}{2} U^2_{0,j} + 5 U^2_{1,j} - 6 U^2_{2,j} + 3 U^2_{3,j} - \frac{1}{2}U^2_{4,j} \right], \label{Dr3r}\\
&\left(Dr\right)_{4r}^2U^2_{1,j} \equiv 
\frac{1}{\Delta r^4} \left[ 2 U^2_{0,j} - 9 U^2_{1,j} + 16 U^2_{2,j} - 14 U^2_{3,j} + 6U^2_{4,j} - U^2_{5,j} \right]. \label{Dr4r}
\end{align}
It can be shown that the DC formula (\ref{BCneuDC4}) is also a fourth order approximation to the Neumann boundary condition and its proof is completely analogous to those performed in Section \ref{ssec:3}. Therefore, it is not included.

We can increase the fourth order discrete approximation (\ref{BCneuDC4}) of the Neumann condition to an arbitrary $p$th order. 
Again, the definition is a natural extension of (\ref{BCneuDC4}), where the continuous derivatives of higher order truncation error terms $\left(u_{3r}\right)_{1j}$, $\left(u_{5r}\right)_{1j},\dots$, $\left(u_{(p-1)r}\right)_{1j}$ 
are approximated by appropriate right one-sided  discrete operators acting on the previously calculated ($p-2$)th ordered numerical solution $U^{p-2}_{i,j}$, approximating the exact solution $u$.
More precisely,
\begin{align}\label{BCneuDCp} 
\frac{{U}_{2,j}-{U}_{0,j}}{2\Delta r}  = -\frac{\partial \uinc}{\partial r} + 
\frac{h^2}{3!}(Dr)_{3r}^{p-2}U^{p-2}_{ij} +
\frac{h^4}{5!}(Dr)_{5r}^{p-4}U^{p-2}_{ij} + \ldots +
\frac{h^{p-2}}{(p-1)!}(Dr)_{(p-1)r}^2U^{p-2}_{ij}. 
\end{align}

The proofs that the finite difference formulas (\ref{BC4pNH})-(\ref{RecurrDiscr2p}) and (\ref{BCneuDCp})  approximate their continuous counterparts to $p$th order are very similar to the proof of Theorem \ref{Th2}. Therefore, they are omitted. The key assumption for these proofs is that the discrete functions $U^{p-2}_{N,j}$,  $F^{p-2}_{l-1,j}$ and $G^{p-2}_{l-1,j}$,
which are obtained in a previous step, are ($p-2$)th order approximations of the continuous solutions $u(R,\theta)$, $F^2_{l-1} (\theta)$, and 
$G^2_{l-1} (\theta)$, respectively.


\section*{Acknowledgments}
The first and third authors acknowledge the support provided by the European Union's Horizon 2020 research and innovation programme under the Marie Sklodowska-Curie grant No 777778 MATHROCKS,
and by the Office of Research and Creative Activities (ORCA) of Brigham Young University. The work of S. Acosta was partially supported by NSF [grant number DMS-1712725]. O. Rojas also thank the European Union’s Horizon 2020 research and innovation programme under the ChEESE project, grant agreement No. 823844, for partially funding this work. 


\bibliographystyle{elsarticle-num}
\bibliography{DeferredCorrection}

\end{document}